\documentclass[11pt]{article}
\usepackage{amssymb}
\usepackage[brazil]{babel}
\usepackage[latin1]{inputenc}
\usepackage{graphicx,color}
\def\nd{\noindent}

\usepackage{color}
\newtheorem{theorem}{Theorem}[section]

\newtheorem{lemma}{Lemma}[section]

\newtheorem{remark}{Remark}[section]

\newtheorem{corollary}{Corollary}[section]
\newcommand{\fim}{\hfill\rule{2mm}{2mm}}
\def\proof{\mbox {\it Proof~}}
\date{}

\begin{document}
\title{
\vspace{0.5in} {\bf\Large Existence and non-existence of Blow-up
solutions for a non-autonomous problem with indefinite and gradient
terms}}

\author{
{\bf Claudianor O. Alves}\footnote{C.O. Alves was partially supported by CNPq/Brazil 303080/2009-4}\\
{\small \textit{Unidade Acad\^emica de Matem\'atica}}\\
{\small \textit{Universidade Federal de Campina Grande}}\\
{\small \textit{58429-900, Campina Grande - PB - Brazil}}\\
{\small \textit{e-mail address: coalves@dme.ufcg.edu.br}}\\
{\bf\large Carlos A. Santos}\footnote{The author acknowledges the support of PROCAD/UFG/UnB  and FAPDF under grant
PRONEX 193.000.580/2009}\,\, {\bf\large and \,\,  Jiazheng
Zhou}\footnote{The author acknowledges the support of CNPq/Brasil.
}\hspace{2mm}\\
{\it\small Departamento de Matem\'atica}\\
{\it\small Universidade de Bras\'ilia}\\
{\it\small 70910-900 Bras\'ilia, DF - Brasil}\\
{\it\small e-mails: csantos@unb.br,
jiazzheng@gmail.com }\vspace{1mm}}
\date{}
\maketitle \vspace{-0.2cm}

\begin{abstract}
\noindent  We deal with existence and non-existence of non-negative
entire solutions that blow-up at infinity for a quasilinear
 problem depending on a non-negative real parameter. Our main objectives in this paper are to
provide  far more general conditions for existence and non-existence
of solutions. To this end, we explore an associated $\mu$-parameter
convective ground state problem, sub and super solutions method
combined and  an approximation arguments to show existence of
solutions. To show the result of non-existence of solutions, we
follow an idea due to Mitidieri-Pohozaev.
\end{abstract}

\nd {\it \footnotesize 2012 Mathematics Subject Classifications}.\\
\nd {\it \footnotesize Key words}: {\scriptsize Quasilinear
equations, Ground state solution, Existence and Non-existence, Large
solutions, Gradient term.}

\section{Introduction}
\def\theequation{1.\arabic{equation}}\makeatother
\setcounter{equation}{0}

\mbox{}

 We consider the problem
\begin{equation}\label{eq1}
\left\{
\begin{array}{c}
\Delta_pu=a(x)f(u)+\mu b(x)|\nabla{u}|^{\alpha}\  \mbox{in}\ \mathbb{R}^N,\\
u\geq0\ \mbox{on}\ \mathbb{R}^N,\
u(x)\stackrel{|x|\to\infty}{\longrightarrow}\infty,
\end{array}
\right.
\end{equation}
where  $N\geq1$, $\alpha \geq 0$ and $ \mu\geq0$ is a real
parameter, $\Delta_p$ is the $p$-Laplacian operator with $1<p<
\infty$, $f:[0,\infty)\rightarrow [0, \infty)$ is a continuous
function such that $f(t)>0$ for $t>0$; $a,b: \mathbb{R}^N\rightarrow
\mathbb{R}$ are continuous functions with $a$ being nonnegative and $b$ can
change of signal.

A solution of (\ref{eq1}) is meant as a nonnegative function in
$C^1(\mathbb{R}^N)$ that satisfies (\ref{eq1}) in distributional
sense. It is well-known as being a entire large (explosive or
blow-up) solutions.

The research by conditions that lead to the existence, non-existence
and behavior asymptotic of solutions for problem (\ref{eq1}), in
bounded domain, mainly without the dependence of the gradient term,
has been much made recently. However for problem (\ref{eq1}) in
whole space, principally with dependance of the gradient term, there
is a less expressive literature.

It is well-known in the mathematical literary that the issue of
existence and non-existence of solution for problem (\ref{eq1}),
without dependance of gradient term, that is, $\mu =0$ in
(\ref{eq1}), are very sensible to the behavior of the potential $a$
at the infinity. If $a=1$ and $f\geq 0$, Keller \cite{keller} and
Osserman \cite{osserman} proved that problem (\ref{eq1}), with
$p=2$, admits a positive solution if only if $f$ satisfies
$$\ \ \ \ \ \ \
\displaystyle\int_{1}^{\infty} {F}(t)^{-{1}/{p}} dt =\infty,~
\mbox{where}~ {F}(t) = \displaystyle \int_{0}^t {f}(s)ds.$$

In 2000, Lair and Wood considered a such $f$, more specifically
$f(u)=u^q$ with $0<q<1$, and showed in \cite{LW2} that problem
(\ref{eq1}), with $p=2$, $\mu =0$ and $a$ is a radially-symmetric
and nonnegative function, admits a solution if only if
$$\displaystyle\int_0^{\infty} r a(r) dr=\infty.$$

In this sense, that is, when the term $f$ satisfies the above
condition, there are a lot of papers studying the issues about
existence and non-existence of solution for (\ref{eq1}) both in
bounded and unbounded domains without or with dependance of gradient
term. See for example, \cite{Hamydy, BZ,GR} and references therein.

In a similar way, when the term $f$
 satisfies
\begin{enumerate}
  \item [(F)] $\ \ \ \ \ \ \ \displaystyle\int_{1}^{\infty} {F}(t)^{-{1}/{p}} dt < \infty$
\end{enumerate}
the looking for by existence of solutions should occurs by
controlling the decaying fast of $a$ at infinity. The above
condition is known as Keller-Osserman condition.

In this sense, Ye and Zhou \cite{ye} proved that a sufficient
condition for existence of solutions for problem (\ref{eq1}) with
$p=2$, $\mu =0$, $f$ a increasing function satisfying $f(0)=0$ and
($F$) is that $a>0$ be a continuous function such that the problem
$$
\leqno(P)~~~~~~~~~~~~~~ \left\{
\begin{array}{c}
-\Delta w=a(x)\ \  \mbox{in}\ \mathbb{R}^N,\\
w>0\ \mbox{on}\ \mathbb{R}^N,\ w(x)\stackrel{|x|
\to\infty}{\longrightarrow}0,
\end{array}
\right.
$$
admits a solution in $C^1(\mathbb{R}^N)$.

On the other hand, for the particular case $f(u)=u^q$ with $q>1$ and
$a$ being a radial continuous function, it was showed by Taliaferro
in \cite{taliaferro} that the existence of solution for ($P$) is
also a necessary condition for the existence of solution for
(\ref{eq1}) with $p=2$. These results show that the solvability of
problem ($P$) is almost a optimal condition for existence of
solution for problem (\ref{eq1}) with $\mu  =0$, $p=2$ and $f$
satisfying (F).

 For this class of problem, that is, (\ref{eq1}) with $\mu=0$
 and $a$ be a non-negative continuous
function, a natural approach to show existence of solution has been
the sub and super solution technique using an argument of
approximation by auxiliary problems defined in balls centered at
origin of $\mathbb{R}^N$ with radius $k=1,2,\cdots$, namely $B_k$.
So, sub and super solutions for (\ref{eq1}) are constructed and some
kind of comparison principle is used to order them.

Now, we are going to do a small overview about results related to
problems like (\ref{eq1}) with $\mu \neq 0$ in bounded domain and
whole space, which in the most have sign-defined potentials. In
1996, Bandle and Giarrusso \cite{BG} proved existence and studied
behavior asymptotic of solutions for
$$
\left\{
\begin{array}{c}
\Delta u=f(u)\pm|\nabla u|^\alpha\ \ \mbox{in}\ \Omega,\\
u\geq0\ \ \ \mbox{on}\ \Omega,\ u(x)\stackrel{d(x) \to
0}{\longrightarrow}\infty,
\end{array}
\right.
$$
where $\Omega\subset \mathbb{R}^N$ is a smooth bounded domain,
$d(x)$ is the distance of $x$ to the boundary of $\Omega$, either
 $f(u)=u^q$ or $f(u)=e^u$, $q>1$ and $\alpha>0$ is a fixe number.

Em 2006, Zhang \cite{ZZ} studied the problems
$$
\left\{
\begin{array}{c}
\Delta u=a(x)f(u)\pm\lambda|\nabla u|^{\alpha}\ \ \mbox{in}~\Omega,\\
u\geq0\ \ \ \mbox{on}\ \Omega,\ u(x)\stackrel{d(x) \to
0}{\longrightarrow}\infty,
\end{array}
\right.
$$
where range interval of $\alpha>0$ depend on sign $\pm$, $a$
behavior like the unique solution of $-\Delta u = 1$ in $\Omega$
with $u=0$ on the boundary of $\Omega$ and $f$ is like $s^q$ at
infinity for some appropriate $q>0$.

In 2011, Huang, Li, Tian and Mu \cite{HL} studied
$$
\left\{
\begin{array}{c}
\Delta u=a(x)f(u)\pm b(x)|\nabla u|^\alpha,\ \ \mbox{in}\ \Omega,\\
u\geq0\ \ \ \mbox{on}\ \Omega,\ u(x)\stackrel{d(x) \to
0}{\longrightarrow}\infty,
\end{array}
\right.
$$
where $\alpha\geq0,\ a,b\in C^\nu(\Omega)$ for some $\nu\in(0,1)$
with $a$ positive and $b$ non-negative functions that can be
singular or null in the boundary of $\Omega$ and $f$ positive is
such $f(s)/s$, $s>0$ is increasing at infinity.

Recently, Hamydy in \cite{H} considered the p-Laplacian operator and
showed the existence of solution for a problem like
$$
\left\{
\begin{array}{cc}
\Delta_p u = a(x)f(u) +b(x)|\nabla u|^{p-1},\ \  \mbox{in}\ \ \Omega,\\
u \geq 0~~ \mbox{on}~~  \Omega,~~~ u(x) \stackrel{d(x) \to
0}{\longrightarrow} \infty,
\end{array}
\right.
$$
where $b\in L^\infty(\Omega)$ can change of sign, $p\geq2$, $f$ is
continuous and increasing with $\inf_{s>0}f(s)/s^q$ is positive for
some $q>p-1$ and $a(x)\geq a_{\infty} >0$, $x \in \Omega$.

In the whole space, there exists a very few papers studying
existence of solutions. In 1999, Lair and Wood \cite{LW} showed the
existence of solutions for the problem
 $$
\left\{
\begin{array}{cc}
\Delta u = a(x) u^q \pm \vert \nabla u \vert^{\alpha}~~ \mbox{in}~~\mathbb{R}^N ,\\
u \geq 0~~ \mbox{on}~~  \Omega,~~~ u(x) \stackrel{d(x) \to
0}{\longrightarrow} \infty.
\end{array}
\right.
$$
For the positive signal, they assumed for instance $0 \leq a(x) \leq
M\vert x \vert^{-2-\beta}$ for big $\vert x \vert$ and either $q
\leq 1+\beta(1-\alpha)/(2-\alpha)$ with $0 < \alpha <1$ or
$\max\{q,\alpha\}>2$, if $\alpha \geq 1$ and for the negative
problem they assumed $a \geq 0$ and with its zero points enclosed by
a bounded surface of non-zero points satisfying
$$\int_0^{\infty}\max_{\vert x \vert = r} a(x)dr <\infty$$
for $N \geq 3$ and $q>\max\{1,\alpha\}$.

Motivated by the above results, Hamydy, Massar and Tsouli
\cite{Hamydy} in 2011 complemented this last result by considering a
more general $\mu$-parameter problem
$$
\left\{
\begin{array}{cc}
\Delta_p u = a(x) f(u) + \mu\vert \nabla u \vert^{p-1}~~ \mbox{in}~~\mathbb{R}^N ,\\
u \geq 0~~ \mbox{in}~~  \mathbb{R}^N,~~~ u(x) \stackrel{|x| \to
\infty}{\longrightarrow} \infty,
\end{array}
\right.
$$
with $\mu \neq 0$, and they proved existence of solutions for $p>2$
and $a(x)\ge a_{\infty},~x\in \mathbb{R}^N$, for some
$a_{\infty}>0$. However, in this case $f$ not satisfies the
Keller-Osserman condition.

For the non-existence of solutions, there exists very few works. In
1999, Mitidieri and Pohozaev \cite{MP} introduced a test-function
method to prove the non-existence of positive solution for
$$\Delta_pu\geq|x|^{-\delta}u^q\ \ \mbox{in}\ \mathbb{R}^N,$$
where $1<p<N,\ q>p-1$ and $p>\delta>1$. For related problems and by
using different techniques, we quote Lair and Wood \cite{LW2} in
2000, Ghergu and Radulescu \cite{GR} in 2004 and references therein.

In a recent paper, Felmer, Quaas and Sirakov \cite{FQS} by using
appropriate super solutions and comparison principles proved the
non-existence of solutions for the autonomous inequality
$$
\Delta u \geq f(u) +g(|\nabla u|)\ \  \mbox{in}\ \  \mathbb{R}^N,\\
$$
where $f$ and $g$ are increasing continuous functions with
$f(0)=g(0)=0$ and either $f$ does not satisfies Keller-Osserman
condition or $g$ satisfies $\int_1^\infty{d s}/{g(s)}<\infty$.

In the above cases, when the potentials $a$ and $b$ are
non-negative, the operator is elliptic uniformly and its
perturbations has $C^1$-regularities, the classical standard
comparison principles, like that in \cite{GT}, have been used to
compare the sub and super solution of (\ref{eq1}), the solutions of
these auxiliary problems each other and these solutions with the sub
and super solutions. So, the solution is built by a diagonal process
limit.

Since, our principal aim in this paper is to consider the
p-Laplacian operator with $1<p<\infty$ and to establish far more
general conditions under potentials $a$ and $b$ (which can be
non-constant and $b$ can be indefinite potential) in the whole
space, the existence and non-existence of solutions for (\ref{eq1})
cannot obtained by standard comparison principles, at least in a
direct way. The principal difficulty is when $b^+ \ne 0$.

To overcome this, we prove a comparison principle for this class of
problem (see theorem 2.1). Besides this, in general the building of
sub and super solution for problems, with dependance of gradient
term, in whole space in general are not easy, principally because we
need obtain the explosive behavior of the solution at infinity.

To get over these difficulties, we show the existence of a
$\mu$-positive ground state solution for an associated
$\mu$-parameter problem with dependance of gradient term which
allows us constructing an super solution for the problem (\ref{eq1})
whose $L^{\infty}(\mathbb{R}^N)$-norm is controlled by the parameter
 (see lemma 2.2).

Concerning to the non-existence of solutions for (\ref{eq1}), a
natural approach to do this is to construct some appropriate radial
super solution for (\ref{eq1}) and apply some comparison principle.
However, this procedure does not work in our case because neither
standard comparison principles nor our result can not be applied.

So, we exploit an idea, due to Mitidieri and Pohozaev \cite{MP}, by
constructing a test function that is null in the exterior of
appropriate balls of $\mathbb{R}^N$.
 By using this test function carefully constructed in
$C_0^{\infty}(\mathbb{R}^N)$ together the infinity-information on
the nonlinearities we get our result after carefully calculations.

 These improve and complement some the prior results of non-existence not only by it does not
to require global information on the terms but also
by it to permit a more class of the nonlinearities
$f$ and potentials $a$ and $b$. We quote the reader principally to
\cite{MP}, \cite{DP}, \cite{LL} and \cite{AS} for whole space and
\cite{H} and \cite{AG} for bounded domain and references therein.

The main contribution of our work is related to the fact that we
present some forms that the terms $a$ and $b$ should interact to
produce existence or non-existence of solutions for (\ref{eq1})
without assuming $f$ is monotonous. In a some sense, these results show
that these interactions are connected with the solvability of a
problem like
$$
\leqno{(P_{\rho})}~~~~~~~~~~~~~~ \left\{
\begin{array}{c}
-\Delta_pw=\rho(x)\ \  \mbox{in}\ \mathbb{R}^N,\\
w>0\ \mbox{on}\ \mathbb{R}^N,\ w(x)\stackrel{|x|
\to\infty}{\longrightarrow}0,
\end{array}
\right.
$$
with $\rho$ given by an appropriate
combination of the potentials $a$ and $b$.

It is well-known that  ($P_{\rho}$) has a $C^1$-solution, if $1<p<N$
and
$$~~~~~~~~~~~~~~
\begin{array}{c}
\displaystyle\int_1^\infty \Big(
t^{1-N}\int_0^tr^{N-1}\hat{\rho}(r)dr\Big)^{\frac{1}{p-1}}dt<\infty
\end{array}
$$
holds, where ${\hat{\rho}}(r)=\max_{|x|=r}\rho(x)$ and $\rho \in
C(\mathbb{R}^N)$ is a non-negative function. In fact, if $p\geq N$,
the problem ($P_{\rho}$) does not have solution for any function
$\rho\geq0$. See for example Serrin and Zou \cite{serin}.

 Now, we state our
principal results. Before this, we need to consider the following
condition.
\begin{enumerate}
  \item [($P$)$_{\rho}$:] Problem $(P_{\rho})$, with $\rho(x)\!=\!\max\{a(x), \!b^+(x)\}$, $x \in \mathbb{R}^N$, admits a super  solution $z$   belonging to
$$(i)~~~~C^1(\mathbb{R}^N),~\mbox{if}~b^+= 0~\mbox{and}~0\leq \alpha \leq p~~~~~~~~~~
(ii)~~~~C^1(\mathbb{R}^N)\cap
W^{1,\infty}(\mathbb{R}^N),~\mbox{if}~b^+\neq 0~\mbox{and}~\alpha =
p-1.~~~~$$
\end{enumerate}
Throughout all this work we are going to denote by
$b^+(x)=\max\{b(x),\ 0\}$ and $b^-(x)=\max\{-b(x),\ 0\}$, $x \in
\mathbb{R}^N$ as being the positive and negative parts of a function
$b$.

\begin{remark}
In $($P$)$$_{\rho}$-$(ii)$, we note that the existence of a
$C^1(\mathbb{R}^N)$-solution of $($P$_{\rho}$$)$ implies its
$W^{1,\infty}(\mathbb{R}^N)$ regularity, if $\rho \in
L^{\infty}(\mathbb{R}^N)$.
\end{remark}

\begin{theorem}\label{teo2} Assume that $\liminf_{t\to\infty}{f(t)}/{t^q}>0$
 for some  $q>\max\{\alpha, p-1, 1\}$, $(P)_{\rho}$ hold
 and $a, b\in L^\infty_{loc}(\mathbb{R}^N)$ with $a$ satisfying 

$$
\leqno{(a_\Omega):}~\mbox{given a  smooth  bounded open set}\
\Omega\subset\mathbb{R^N}, \mbox{there exists}\ a_\Omega>0\
\mbox{such that}\ a(x)\geq a_\Omega\ \mbox{a.a.}\ \mbox{in}\ \Omega.
$$
Then there exists $0< \mu^* \leq \infty$ such that the problem
$(\ref{eq1})$ admits at least one solution for each $0 \leq \mu <
\mu^*$ given. Besides this, $\mu^*=\infty$, if
$($P$)$$_{\rho}\!\!-\!\!(i)$ holds.
\end{theorem}
\smallskip

In the sequel, we are interested in considering either $\alpha>p-1$
or potentials $a$ and $b$ such that the problem (\ref{eq1}) has no
sub solution in $C^1(\mathbb{R}^N)$. More
specifically, we will consider the problem
\begin{equation}\label{non}
 \left\{
\begin{array}{c}
 \Delta_pu \geq a(x)f(u)+b(x)|\nabla u|^\alpha,\ \mbox{in}\ \mathbb{R}^N ,\\
u \geq 0 ~\mbox{on}~ \mathbb{R}^N , ~~
u(x)\stackrel{|x|\to\infty}{\longrightarrow}\infty,
\end{array}
\right.
\end{equation}
where $a,b:\mathbb{R}^N \to \mathbb{R}$ are
$L^{\infty}_{loc}(\mathbb{R}^N)$ nonnegative
functions and $f:(0,\infty)\to[0,\infty)$ is an appropriate
function. We are going to denote by $B_R$ the ball centered at
origin of $\mathbb{R}^N$ with radius $R>0$.

\begin{theorem}\label{teo4} Assume one of the below case holds for some $R_0>0$:
\begin{enumerate}
  \item [$(i)$] $a,b>0$ $a.a.$ on $\mathbb{R}^N \setminus
  B_{R_0}$, $\liminf_{t\to\infty}{f(t)}/{t^q}>0$ for some $q>p-1$
  and either
\begin{enumerate}
  \item [$(i_1)$] $~~ \displaystyle\limsup_{R\to\infty}R^{\frac{pq}{\theta-q}}\int_{R\leq|x|\leq 2R}a(x)^{\frac{\theta}{\theta-q}}dx<\infty\  \mbox{for some}\
  \theta\in({p-1},{q})~~\mbox{or}
  $
  \item [$(i_2)$] $~~\displaystyle\limsup_{R\to\infty}R^{\frac{\alpha}{p-1-\alpha}}\int_{R\leq|x|\leq 2R}b(x)^{\frac{p-1}{p-1-\alpha}}dx<
  \infty~\mbox{with}~\alpha>p-1.$
\end{enumerate}
  \item [$(ii)$] $a,b>0$ $a.a.$ on $\mathbb{R}^N \setminus
  B_{R_0}$,  $\liminf_{t\to\infty}{f(t)}/{t^q}>0$ for some $q>0$,
  $\alpha>p-1$
  and
  $$\limsup_{R\to\infty}R^{\frac{\alpha}{\theta-\alpha}}\int_{R\leq|x|\leq 2R}[a(x)^{\frac{\theta}{p-1}-1} b(x)]^{\frac{p-1}{\theta-\alpha}}dx
<\infty~\mbox{for some}~\theta\in(p-1,\alpha).$$
\end{enumerate}
Then problem $(\ref{non})$  has no solution in $C^1(\mathbb{R}^N)$.
\end{theorem}

\section{\textbf{Auxiliary Results}}

In this section, we are going to present some very important results
in our approach. In first place, we are going to consider the
inequalities
\begin{equation}\label{subb}
\int_\Omega |\nabla u|^{p-2}\nabla u\nabla \varphi dx+\int_\Omega
[a(x)h(u)+\mu b(x)|\nabla{u}|^\alpha]\varphi dx\leq0,
\end{equation}
and
\begin{equation}\label{supp}
\int_\Omega |\nabla v|^{p-2}\nabla v\nabla \varphi dx+\int_\Omega
[a(x)h(v)+\mu b(x)|\nabla{v}|^\alpha]\varphi dx\geq0,
\end{equation}
for all $\varphi\geq0,\ \varphi\in C_0^\infty(\Omega),$ where
$\Omega \subset \mathbb{R}^N$ is a smooth bounded domain,
$\alpha\geq p-1$, $h:\mathbb{R}\to\mathbb{R}$ is a increasing
continuous function and $a,b\in L_{loc}^{\infty}(\Omega)$ with $a$
satisfying 
$$
\leqno{(a_\Omega)^{\prime}:}~~~\mbox{given a smooth open set}\
O\subset\subset\Omega~\mbox{there exists}\ a_o>0\ \mbox{such that}\
a(x)\geq a_o\ a.a.\ \mbox{in}\ O.
$$
Before proving our first result in this section, we state the below
lemma, whose proof is easy.

\begin{lemma}\label{lema0}
Assume $\alpha \geq 0$. Then for each $\tau>1$ given, there exists a
$\nu=\nu(\tau)>0$ such that
$$
(i)~~t^{\alpha}-1\leq \nu(t-1)^{\alpha},~t\geq \tau~~~~(ii)~~\vert t^{\alpha}-1\vert \leq \tau^{\alpha}-1,~\tau^{-1}<t<\tau.
$$
\end{lemma}

\begin{theorem}\label{prop2} \textbf{(A comparison Principle)} \, Assume $a,b$ and $h$ like above. If $u,v\in W^{1,\infty}_{loc}(\Omega)\cap C(\Omega)$
satisfy $(\ref{subb})$ and $(\ref{supp})$ respectively, and
$$
\lim_{x\to y}(u(x)-v(x)) \in [-\infty,0],~\mbox{for each}~y \in
\partial\Omega,
$$
 then  $u\leq v$ in $\Omega$.
\end{theorem}

\noindent \textbf{Proof.} In what follows, we argue  by contraction.
Assume that $\omega(x)=u(x)-v(x),\ x\in\Omega$ is such that
$\overline{\varepsilon}=\sup_{\Omega}\omega(x)>0.$ So, for
$\varepsilon\in({\overline{\varepsilon}}/{2},\overline{\varepsilon})$
given, the function $\omega_\varepsilon$ defined by
$\omega_\varepsilon=\max\{0, \omega-\varepsilon\}$ is not null
precisely in
$$\Omega_\varepsilon:=\{x\in\Omega,\
\varepsilon<\omega(x)\leq\overline{\varepsilon}\}.$$ Besides this,
we have
\begin{equation}\label{sz0}
\Omega_{\varepsilon_2} \subset \Omega_{\varepsilon_1} \subset
\Omega_{{{\overline{\varepsilon}}/{2}}},~\mbox{for}~{\overline{\varepsilon}}/{2}<\varepsilon_1<\varepsilon_2<\overline{\varepsilon}
\end{equation}
and
$$
\Omega_\varepsilon\subset\subset\Omega, \,\,\,\, \mbox{that is}, \,\,\,\, \overline{\Omega}_\varepsilon \,\,\, \mbox{is a compact set in} \,\,\, \Omega.
$$

As $\omega_\varepsilon\in W_0^{1,p}(\Omega)$ and
$\omega_\varepsilon\geq0$, we can use it as test function in,
(\ref{subb}) and (\ref{supp}) to obtain
$$\int_{\Omega_\varepsilon}\{|\nabla u|^{p-2}\nabla u-|\nabla v|^{p-2}\nabla v\}\nabla \omega_\varepsilon dx\leq\int_{\Omega_\varepsilon} \{a(x)[h(v)-h(u)]+
\mu b(x)[|\nabla v|^\alpha-|\nabla u|^\alpha]\}\omega_\varepsilon
dx.$$

So, by a classical inequality,
$$c_pM_{\varepsilon}\int_{\Omega_\varepsilon}|\nabla \omega_\varepsilon|^pdx\leq\int_{\Omega_\varepsilon}\{|\nabla u|^{p-2}\nabla u-|\nabla v|^{p-2}\nabla v\}
\nabla \omega_\varepsilon dx,$$ where
$$
 M_\varepsilon:=\left\{
\begin{array}{l}
 {(|\nabla v|_{L^\infty(\Omega_\varepsilon)}+|\nabla
u|_{L^\infty(\Omega_\varepsilon)}+1)^{p-2}},\ \mbox{if}\ 1<p\leq
2,\\
\\
1,\ \mbox{if}\ p\geq 2
\end{array}
\right.
$$
and $c_p$ is a positive constant that it does not depends on
$\varepsilon$. In particular, from (\ref{sz0}),
\begin{equation}\label{sz00}
0<M_{{\overline{\varepsilon}}/{2}} \leq M_\varepsilon \leq
1,~\mbox{for
all}~\varepsilon\in({\overline{\varepsilon}}/{2},\overline{\varepsilon}).
\end{equation}
Hence,
\begin{equation}\label{sz2}
{c}_p
M_\varepsilon\int_{\Omega_\varepsilon}|\nabla\omega_\varepsilon|^pdx\leq\int_{\Omega_\varepsilon}\{a(x)[h(v)-h(u)]+\mu
|b(x)|||\nabla v|^\alpha-|\nabla u|^\alpha|\} \omega_\varepsilon dx.
\end{equation}

Now, given an $\tau>1$, we shall consider the ensuing subsets of
$\Omega_\varepsilon$
\[
G(\tau)=\{x\in\Omega_\varepsilon,\ \nabla u\neq\nabla v,\
|\nabla v|\geq \tau|\nabla u|\}
\]
\[
\tilde{G}(\tau)=\{x\in\Omega_\varepsilon,\ \nabla u\neq\nabla v,\
|\nabla v|\leq
\frac{1}{\tau}|\nabla u|\},
\]
\[
L(\tau)=\{x\in\Omega_\varepsilon,\ \nabla u\neq\nabla v,\
\frac{1}{\tau}|\nabla
u|<|\nabla v|<\tau|\nabla u|\},
\]
and
\[
I(\tau)=\{x\in\Omega_\varepsilon,\ \nabla u=\nabla v\}.
\]
Using (\ref{sz2}) together with the monotonicity of $h$ in $I(\tau)$
and the above sets, we get
\begin{equation}\label{sz3}
\begin{array}{lll}
 {c}_p M_\varepsilon \displaystyle\int_{\Omega_\varepsilon}|\nabla\omega_\varepsilon|^pdx&\leq&\displaystyle\int_{\Omega_\varepsilon}\{a(x)[h(v)-h(u)]+\mu |b(x)|||\nabla v|^\alpha-|\nabla u|^\alpha|\}
\omega_\varepsilon dx\\
\\
&\leq&\displaystyle\int_{G(\tau)}\mu |b(x)|||\nabla v|^\alpha-|\nabla
u|^\alpha|\omega_\varepsilon dx+\int_{\tilde{G}(\tau)}\mu
|b(x)|||\nabla v|^\alpha-
|\nabla u|^\alpha|\omega_\varepsilon dx\\
\\
&+&\displaystyle\int_{L(\tau)}\mu |b(x)|||\nabla v|^\alpha-|\nabla
u|^\alpha|\omega_\varepsilon
dx-\int_{L(\tau)}a(x)[h(u)-h(v)]\omega_\varepsilon dx.\nonumber
\end{array}
\end{equation}
Now, by Lemma \ref{lema0} and (\ref{sz3}),
$$
\begin{array}{lll}
 {c}_p M_\varepsilon \displaystyle\int_{\Omega_\varepsilon}|\nabla\omega_\varepsilon|^pdx&
\leq&\mu\nu\displaystyle\int_{G(\tau)}|b(x)|||\nabla v|-|\nabla
u||^\alpha\omega_\varepsilon
dx+\mu\nu\int_{\tilde{G}(\tau)}|b(x)|||\nabla v|-
|\nabla u||^\alpha\omega_\varepsilon dx\\
\\
&+&\displaystyle\int_{L(\tau)}\mu|b(x)|(\tau^\alpha-1)|\nabla
u|^\alpha\omega_\varepsilon
dx-\int_{L(\tau)}a(x)[h(u)-h(v)]\omega_\varepsilon dx.
\end{array}
$$

Since, $h$ is increasing continuous, we have
$$h(u(x))-h(v(x))\geq
h(v(x)+{\overline{\varepsilon}}/{2})-h(v(x)):=\sigma_\varepsilon\
\mbox{in}\ \Omega_\varepsilon,$$ where
$\sigma_\varepsilon:=\min_{\Omega_\varepsilon}
[h(v(x)+{\overline{\varepsilon}}/{2})-h(v(x))] >0$. Thus, by using
the hypothesis $(a_\Omega)^{\prime}$, there exists
$\tau_{\varepsilon}>1$, enough near of $1$, such that
$$
\mu|b(x)|(\tau_\varepsilon^\alpha-1)
|\nabla u|^\alpha_{L^\infty(\Omega_\varepsilon)}-a(x)[h(u)-h(v)]\leq
\mu|b|_{L^\infty(\Omega_\varepsilon)}(\tau_\varepsilon^\alpha-1)|\nabla
u|^\alpha_{L^\infty(\Omega_\varepsilon)}-a_{\Omega_\varepsilon}\sigma_\varepsilon<
0
$$
in  $L(\tau_{\varepsilon})$. Hence,
$$
{c}_p M_\varepsilon
\displaystyle\int_{\Omega_\varepsilon}|\nabla\omega_\varepsilon|^pdx
\leq \mu\nu\displaystyle\int_{G(\tau_{\varepsilon})}|b(x)|||\nabla
v|-|\nabla u||^\alpha\omega_\varepsilon
dx+\mu\nu\int_{\tilde{G}(\tau_{\varepsilon})}|b(x)|||\nabla v|-
|\nabla u||^\alpha\omega_\varepsilon dx\
$$
from where it follows that
$$
{c}_p M_\varepsilon
\displaystyle\int_{\Omega_\varepsilon}|\nabla\omega_\varepsilon|^pdx
\leq \mu\nu\displaystyle\int_{\Omega_\varepsilon}|b(x)|||\nabla
v|-|\nabla u||^\alpha\omega_\varepsilon dx \leq
\mu\nu\displaystyle\int_{\Omega_\varepsilon}|b(x)||\nabla
\omega_\varepsilon|^\alpha\omega_\varepsilon dx
$$
and so,
\begin{equation}\label{sz6}
\begin{array}{lll}
 {c}_p  M_\varepsilon\displaystyle\int_{\Omega_\varepsilon}|\nabla \omega_\varepsilon|^pdx&\leq&\mu\nu|b|_{L^\infty(\Omega_\varepsilon)}
 \displaystyle\int_{\Omega_\varepsilon}|\nabla \omega_\varepsilon
|^\alpha|\omega_\varepsilon|dx\\
\\
&\leq&\mu\nu|b|_{L^\infty(\Omega_\varepsilon)}\displaystyle\int_{\Omega_\varepsilon}|\nabla
\omega_\varepsilon|^{p-1}|\nabla\omega_\varepsilon|^{\alpha-p+1}|
\omega_\varepsilon|dx\\
\\
&\leq& \mu\nu|b|_{L^\infty(\Omega_\varepsilon)}d_\varepsilon\displaystyle\int_{\Omega_\varepsilon}|\nabla\omega_\varepsilon|^{p-1}|\omega_\varepsilon|dx,\\
\\
&\leq&\mu\nu|b|_{L^\infty(\Omega_\varepsilon)}
d_\varepsilon\left(\displaystyle\int_{\Omega_\varepsilon}|\nabla
\omega_\varepsilon|^p\right)^{\frac{p-1}{p}}\left(\displaystyle\int_{\Omega_\varepsilon}
|\omega_\varepsilon|^p\right)^{\frac{1}{p}}dx\\
\\
&\leq&\mu\nu|b|_{L^\infty(\Omega_\varepsilon)}
d_\varepsilon\left(\displaystyle\int_{\Omega_\varepsilon}|\nabla
\omega_\varepsilon|^p\right)^{\frac{p-1}{p}}\left(\displaystyle\int_{\Omega_\varepsilon}|\omega_\varepsilon|^{p^*}\right)^{\frac{1}{p^*}}
med(\Omega_\varepsilon)^{\frac{1}{N}},
\end{array}
\end{equation}
where
$d_\varepsilon=|\nabla\omega_\varepsilon|^{\alpha-p+1}_{L^\infty(\Omega_\varepsilon)}$
and $med(\Omega_\varepsilon)$ is the  measure of Lebesgue of
$\Omega_\varepsilon$. Again, from (\ref{sz0}),
\begin{equation}\label{sz5}
0<d_\varepsilon \leq
d_{\overline{\varepsilon}/2}~~\mbox{for}~~\overline{\varepsilon}/2<\varepsilon<\overline{\varepsilon}.
\end{equation}
Using the Sobolev imbedding, we know that
$$\left(\int_{\Omega_\varepsilon}|\omega_\varepsilon|^{p^*}\right)^{\frac{1}{p^*}}\leq
d\left(\int_{\Omega_\varepsilon}|\nabla
\omega_\varepsilon|^p\right)^ {\frac{1}{p}},
$$
where $d>0$ is a constant not depending of $\varepsilon$. This
combined with (\ref{sz6}) gives
$${c}_p M_\varepsilon \displaystyle\int_{\Omega_\varepsilon}|\nabla\omega_\varepsilon|^pdx\leq
\mu\nu|b|_{L^\infty(\Omega_\varepsilon)} d_\varepsilon
d\left(\int_{\Omega_\varepsilon}|\nabla\omega_\varepsilon|^p\right)^{\frac{p-1}{p}}
\left(\int_{\Omega_\varepsilon}|\nabla\omega_\varepsilon|^p\right)^{\frac{1}{p}}med(\Omega_\varepsilon)^{\frac{1}{N}},
$$
that is,
$$
 1\leq\displaystyle\frac{\mu\nu|b|_{L^\infty(\Omega_\varepsilon)} d_\varepsilon d}{{c}_p M_\varepsilon}med(\Omega_\varepsilon)^{\frac{1}{N}}.
$$
Thus, by (\ref{sz0}), (\ref{sz00}) and (\ref{sz5}) together with $
|b|_{L^\infty(\Omega_\varepsilon)}\leq
|b|_{L^\infty(\Omega_{\overline{\varepsilon}/2})}$, we get
$$
1\leq\displaystyle\frac{\mu\nu|b|_{L^\infty(\Omega_{\overline{\varepsilon}/2})}
d_{\overline{\varepsilon}/2}d}{{c}_p
M_{\overline{\varepsilon}/2}}med(\Omega_\varepsilon)^{\frac{1}{N}}.
$$
Once that $med(\Omega_\varepsilon)\to 0$ as
$\varepsilon\to\overline{\varepsilon}$, we obtain a contradiction.
Therefore, this proves the theorem.

\begin{lemma}
\label{lema1} Suppose that  $\eta<0$, $\alpha \geq 0$ and
$(P)_{\rho}\!-\!(ii)$ holds. Then, there exist $0 < \Lambda_*<
\infty$ and $\omega=\omega_{\mu}\in C^1(\mathbb{R}^N)$ satisfying
$$
\left\{
\begin{array}{c}
 -\Delta_p\omega\geq a(x)[1+(\omega(x)+1)^\eta/2]+\mu b^+(x)|\nabla\omega|^{\alpha},\ \  \mbox{in}~~\ \mathbb{R}^N ,\\
\omega>0~~  \mbox{in}~~ \mathbb{R}^N ,~~  \omega\stackrel{|x|
\to\infty}{\longrightarrow}0,
\end{array}
\right.
$$
for each $0 \leq \mu < \Lambda_*$ given. Besides this, if $0 \leq
\alpha <p-1$, then $ \Lambda_*=\infty$.
\end{lemma}
\noindent \textbf{Proof.} First of all, we let
$h(s)=2+s^\eta$ for $s\geq 0$ and
$$
F(s)={s^2}/{\int_0^s\frac{t}{h(t)^{{1}/({p-1})}}dt},~s>0.
$$
We point out that $F(s)^{p-1}\geq h(s)$  and ${F(s)}/{s}$ is a
non-increasing continuous function in $(0, +\infty)$.

So, we have well-defined the function
$$H(\tau)=\frac{1}{\tau}\int_0^\tau\frac{t}{F(t)}dt-\frac{1}{\tau}\int_0^1\frac{t}{F(t)}dt, \tau \geq 1$$
with $H(1)=0$. Since,
$$
\begin{array}{lll}
\displaystyle \frac{1}{\tau}\int_0^\tau\frac{s}{F(s)}ds&\geq&
\displaystyle\frac{1}{\tau}\int_{{\tau}/{2}}^\tau\frac{s}{F(s)}ds\geq\frac{1}{2}
\frac{{\tau}/{2}}{F({\tau}/{2})}\\
\\
&\geq&\displaystyle
\frac{{\tau}/{2}}{8h({\tau}/{4})^{{1}/({p-1})}}\to+\infty~~\mbox{as}~~\tau
\to \infty,
\end{array}
$$
it follows that $\lim_{\tau\to\infty}H(\tau)=\infty$.

Thus, there exists a $\tau_\infty>0$ such that
$$\frac{1}{\tau_\infty}\int_0^{\tau_\infty}\frac{t}{F(t)}dt>\Vert z\Vert_\infty+\frac{1}{\tau_\infty}\int_0^1\frac{t}{F(t)}dt,$$
where $z\in C^1(\mathbb{R}^N)\cap W^{1,\infty}(\mathbb{R}^N)$ is
giving by the hypothesis $(P)_\rho\!-\!(ii)$.

After this, we can define a function $v\in C^1(\mathbb{R}^N)$ by
\begin{equation}
\label{taui}
z(x)+\frac{1}{\tau_\infty}\int_0^1\frac{t}{F(t)}dt=\frac{1}{\tau_\infty}\int_0^{v(x)+1}\frac{t}{F(t)}dt,~x
\in \mathbb{R}^N \end{equation} and infer that $1 \leq
v(x)+1<\tau_\infty$ for all $x\in\mathbb{R}^N$ and $v(x)\to0$ when
$|x|\to\infty$. Moreover, by a direct computing, we also have
$$
\begin{array}{ccl}
 \displaystyle\int_{\mathbb{R}^N}|\nabla v|^{p-2}\nabla v\nabla\varphi dx&=&\displaystyle\int_{\mathbb{R}^N}\tau_\infty^{p-1}|\nabla z|^{p-2}\nabla z
\nabla\Big(\frac{F(v(x)+1)^{p-1}}{[v(x)+1]^{p-1}}\varphi\Big)dx\\
\\
&-&\displaystyle\tau_\infty^{p}(p-1)\int_{\mathbb{R}^N}\Big(\frac{F(v(x)+1)}{v(x)+1}\Big)^{p-1}\Big(\frac{F(s)}{s}\Big)_{\mid_{(v(x)+1)}}'
|\nabla z|^p\varphi dx\\
\\
&\geq&\tau_\infty^{p-1}\displaystyle\int_{\mathbb{R}^N}\frac{F(v(x)+1)^{p-1}}{[v(x)+1]^{p-1}}\rho(x)\varphi dx\\
\\
&\geq&\displaystyle\int_{\mathbb{R}^N}F(v(x)+1)^{p-1}\rho(x)\varphi
dx\geq\int_{\mathbb{R}^N}\rho(x)h(v(x)+1)\varphi dx\\
\\
&\geq&\displaystyle\int_{\mathbb{R}^N}a(x)(1+(v+1)^{\eta}/2)\varphi
dx +
\displaystyle\frac{1}{2}\int_{\mathbb{R}^N}b^+(x)h(v(x)+1)\varphi
dx.
\end{array}
$$
Since,
$$
\begin{array}{lcl}
&&\displaystyle\int_{\mathbb{R}^N}b^+(x)h(v(x)+1)\varphi dx \geq \displaystyle\int_{\mathbb{R}^N}b^+(x)\Vert\nabla v \Vert^{-\alpha}_{\infty} \vert\nabla v \vert^{\alpha}\varphi dx\\
\\
&\geq&\displaystyle\frac{1}{\tau_\infty^{-
\alpha}}\int_{\mathbb{R}^N}\Big(\frac{v(x)+1}{F(v(x)+1)}\Big)^{\alpha}
\Vert\nabla z \Vert_\infty^{-{\alpha}}b^+(x)|\nabla
v|^{\alpha}\varphi dx,
\end{array}
$$
it follows by monotonicity of $F(s)/s$, $s \geq 0$ that
$$ \int_{\mathbb{R}^N}|\nabla v|^{p-2}\nabla v\nabla\varphi dx\geq
\int_{\mathbb{R}^N}[a(x)[1+(v(x)+1)^\eta/2]+\mu b^+(x)|\nabla
v|^{\alpha}]\varphi dx$$ for all $0\leq\varphi\in
C_0^\infty(\mathbb{R}^N)$ and $0 \leq \mu \leq
\Lambda_*:=[\tau_\infty{F(1)}\Vert\nabla z
\Vert_\infty]^{-{\alpha}}>0$ given.

Beside this, if $0 \leq\alpha<p-1$, then for each $0
\leq \mu <\infty$ given, we define $\varpi(x)=\theta v(x)$,
$x\in\mathbb{R}^N$, where $v\in C^1(\mathbb{R}^N)$ is given by
(\ref{taui}) and
$\theta=\max\{1,({\mu}/{\Lambda_*})^{{1}/({p-1-\alpha})}\}$.

So, computing we have
$$
\begin{array}{lll}
 \displaystyle\int_{\mathbb{R}^N}|\nabla \varpi|^{p-2}\nabla \varpi\nabla\varphi dx&=& \displaystyle\theta^{p-1}\int_{\mathbb{R}^N}|
\nabla v|^{p-2}\nabla v\nabla\varphi dx\\
&\geq&\displaystyle\theta^{p-1}\int_{\mathbb{R}^N}[a(x)(1+\frac{1}{2}(v(x)+1)^\eta)+\Lambda_*b^+(x)|\nabla
v|^{\alpha}]\varphi dx.
\end{array}
$$

Now, it follows from definition of $\theta$ and $\eta <0$ that
$$
\begin{array}{lll}
 1+\frac{1}{2}(\varpi+1)^\eta&=& 1+\frac{1}{2}(\theta v+1)^\eta\leq1+\frac{1}{2}(v+1)^\eta\leq
\theta^{p-1}(1+\frac{1}{2}(v+1))^\eta
\end{array}
$$
and
$$
\begin{array}{lll}
\mu b^+(x)|\nabla \varpi|^\alpha&=&\mu b^+(x)\theta^\alpha|\nabla
v|^\alpha\leq\theta^{p-1}\Lambda_*b^+(x)|\nabla v|^\alpha.
\end{array}
$$

That is,
$$\int_{\mathbb{R}^N}|\nabla \varpi|^{p-2}\nabla v\nabla\varphi dx\geq \int_{\mathbb{R}^N}[a(x)(1+\frac{1}{2}(\varpi(x)+1)^\eta)+
\mu b^+(x)|\nabla \varpi|^\alpha]\varphi dx.$$ This ends our proof.

\section{Existence of solution for (\ref{eq1}) in bounded domain}

In this section, our main goal is proving the existence
of solution for the problem
\begin{equation}\label{3.1}
\left\{
\begin{array}{c}
\Delta_pu=a(x)f(u)+\mu b(x)|\nabla{u}|^\alpha\ \ \ \mbox{in}\ \Omega,\\
u\geq0\ \ \ \mbox{in}\ \Omega,\ u(x)\stackrel{d(x)\to
0}{\longrightarrow}\infty,
\end{array}
\right.
\end{equation}
where $\Omega\subset\mathbb{R}^N$ is a smooth bounded domain, $a,b:
\Omega\rightarrow \mathbb{R}$ are suitable functions with $a\geq0$,
$f: [0,\infty)\rightarrow [0, \infty)$ is a continuous function with
$f(0)=0$, $0 \leq \alpha \leq p,\ \mu\geq0\ \mbox{is a real
parameter}$ and $N\geq1$.

To do this, we need to show the next result.

\begin{lemma}
\label{lema31} Assume that $h \in L^{\infty}(\Omega)$ is a
nonnegative function and $0 \leq \alpha \leq p$ with $p>1$. Then
\begin{equation}\label{zepsilon}
 \left\{
\begin{array}{c}
-div((|\nabla u|^{p-2}+\epsilon)\nabla u)=\mu h(x)(|\nabla u|+1)^\alpha\ \mbox{in}\ \Omega,\\
u\geq0\ \mbox{in}\ \Omega\ \ u=0\ \mbox{on}\ \partial\Omega
\end{array}
\right.
\end{equation}
admits a solution $u=u_{\varepsilon,\mu}\in C^1(\overline{\Omega})$
for each $0 \leq \varepsilon < 1$ and $0\leq\mu <\Lambda^*$ given,
for some $\Lambda^*=\Lambda^*(\Omega)>0$. Besides this, $\Vert
u_{\varepsilon,\mu} \Vert_{\infty} \leq C$ not depending on
$\varepsilon>0$.
\end{lemma}
\begin{proof} First, we note that for each $h \in L^{\infty}(\Omega)$, it follows by theorem of Browder-Minty that there exists a unique
  $\omega_\epsilon\in W_0^{1,p}(\Omega)$ solution of the problem
\begin{equation}\label{omegaepsilon}
 \left\{
\begin{array}{c}
-div((|\nabla u|^{p-2}+\epsilon)\nabla u)=h(x)\ \mbox{in}\ \Omega,\\
u\geq0\ \mbox{in}\ \Omega\ \ u(x)=0\ \mbox{on}\ \partial\Omega.
\end{array}
\right.
\end{equation}
Besides this, taking $-\omega_\epsilon^-$ as a test function, we get
$\omega_\epsilon \geq 0$, since $h \geq 0$.
\smallskip

\noindent {\bf Claim:} $\omega_\epsilon \in L^{\infty}(\Omega)$ and
$\Vert\omega_\epsilon \Vert_\infty\leq C$ for some $C>0$, which does not
depend of $\epsilon>0$.\\
\smallskip

\noindent In fact, first we note that using $ \epsilon \geq 0 $,
$\omega_\epsilon$ as test function and the Sobolev embedding, we
have
\begin{equation}\label{hinfty}
\Vert\omega_\epsilon \Vert_{1,p}\leq C\Vert
h\Vert_\infty^{{1}/{(p-1)}}~\mbox{for some}~C>0.
\end{equation}
So, if $p\geq N$, we get by using Sobolev embedding again that
$\Vert\omega_\epsilon\Vert_\infty \leq C\Vert
h\Vert_\infty^{{1}/{(p-1)}}$.

Now, if $1<p<N$, we are going to denote by $S>0$ the best constant
of the inequality of Sobolev-Poincaré and
 let $L=\Vert h \Vert_\infty^{{1}/{p}}S$. Following the arguments in  \cite{Moser}, we define the increasing
 sequence $(\gamma_{k})$ with $\gamma_{1}>1$, $\gamma_{k}  \stackrel{k\to
\infty}{\longrightarrow} \infty$, $\gamma^*_k $ as
 $$\gamma_1=p^*,\ \
\gamma^*_k=\gamma_k-1+p,\ \
 \gamma_{k+1}={\gamma^*_kp^*}/{p}$$
 and
$$L_1=\Vert\omega_\epsilon\Vert_{p^*}:=\Vert\omega_\epsilon\Vert_{L^{p^*}(\Omega)},\ \ L_{k+1}=L^{\frac{p}{\gamma^*_k}}\gamma_k^{-\frac{1}{\gamma_k^*}}
\Big(\frac{\gamma_k^*}{p}\Big)^{\frac{p}{\gamma_k^*}}L_k^{\frac{\gamma_k}{\gamma_k^*}},$$
where $p^*=pN/(N-p)$, if $1<p<N$ and
$L_1=\Vert\omega_\epsilon\Vert_{p^*} \leq C\Vert
h\Vert_\infty^{{1}/{(p-1)}}$  by using (\ref{hinfty}) together with
Sobolev embedding.

As a consequence of this, we can prove, by a induction process, that
\begin{equation}
\label{1.24}
 \Vert \omega_\epsilon \Vert_{\gamma_k}\leq L_k\ \mbox{for all}\ k,
 ~\mbox{where}~\Vert \cdot \Vert_{\gamma_k}:=\Vert \cdot \Vert_{L_{\gamma_k}(\Omega)}.
 \end{equation}
To do this, we are going to  consider a  $\psi_n \in
C^1([0,\infty))$ such that
 $0\leq\psi'_n(t)\leq 1$, $\psi_n(t)= t,\ |t|\leq n$ and $\psi_n(t)= n+2,\ |t|\geq
 n+2$ for each $n\in\mathbb{N}$ and to define
 $u_n=\psi_n(\omega_\epsilon)$. So we have
 $0 \leq u_n\leq\omega_\epsilon$ in $\Omega$ and
 $u_n^{l}\in W_0^{1,p}(\Omega)\cap L^\infty(\Omega)$ for each $l\in [1,\infty)$.

Now, by induction hypothesis, we have
$$\int_\Omega h(x)u_n^{\gamma_k}dx\leq \Vert h\Vert_\infty \Vert u_n\Vert_{\gamma_k}^{\gamma_k}\leq\Vert h\Vert_\infty\Vert\omega_\epsilon\Vert_{\gamma_k}^
{\gamma_k}\leq\Vert h\Vert_\infty L_k^{\gamma_k}$$ and by
definitions of $\psi_n$,  $(\gamma_{k})$ and $(\gamma^*_k)$, we have
$$
\begin{array}{lcl}
 &&\!\!\!\!\!\!\!\!\!\!\gamma_k\displaystyle\int_\Omega(|\nabla\omega_\epsilon|^{p-2}+\epsilon)|\nabla\omega_\epsilon|^2\psi'_n(\omega_\epsilon)u_n^{\gamma_k-1}dx\geq
\gamma_k\displaystyle\int_\Omega|\nabla\omega_\epsilon|^p\psi'_n(\omega_\epsilon)u_n^{\gamma_k-1}dx\\
&\geq&\gamma_k\displaystyle\int_\Omega|\nabla
u_n|^pu_n^{\gamma_k-1}dx=\gamma_k\Big(\frac{p}{\gamma_k^*}\Big)^p\int_\Omega|\nabla(u_n^{\frac{\gamma_k^*}{p}})|^pdx\geq
S^{-p}\gamma_k\Big(\frac{p}{\gamma_k^*}\Big)^p\Vert
u_n^{\frac{\gamma_k^*}{p}}\Vert^p_{p^*}.
\end{array}
$$

So, using $u_n^{\gamma_k}$ as a test function in
(\ref{omegaepsilon}), it follows

$$|u_n^{\frac{\gamma_k^*}{p}}|_{p^*}^p\leq S^p\gamma_k^{-1}\Big(\frac{\gamma_k^*}{p}\Big)^p\Vert h\Vert_\infty L_k^{\gamma_k},$$
that is, by definition of $(\gamma_{k})$ and $(\gamma^*_k)$, we have

$$\Vert u_n \Vert^{\gamma_k^*}_{\gamma_{k+1}}\leq S^p\gamma_k^{-1}\Big(\frac{\gamma_k^*}{p}\Big)^p\Vert h\Vert_\infty L_k^{\gamma_k}=L^p\gamma_k^{-1}
\Big(\frac{\gamma_k^*}{p}\Big)^pL_k^{\gamma_k}=
L_{k+1}^{\gamma_k^*}.$$

Now, doing $n \to \infty$, we get  $\Vert \omega_\epsilon
\Vert_{\gamma_{k+1}}\leq L_{k+1}$. This proves (\ref{1.24}).

Below, we are going to show that $(L_{k})$ is bounded. To do this we
are going to define $(E_k)$ as $E_{k}=\gamma_{k}\ln L_{k}$. So,
$$
\begin{array}{ccl}
 E_{k+1} &=&\frac{\gamma_k^*p^*}{p}\Big[\frac{p}{\gamma_k^*}\ln
L-\frac{1}{\gamma_k^*}\ln\gamma_k+\frac{p}{\gamma_k^*}\ln
\gamma_k^*-
\frac{p}{\gamma_k^*}\ln p+\frac{\gamma_k}{\gamma_k^*}\ln L_k\Big]\\
\\
&\leq&\frac{\gamma_k^*p^*}{p}\Big[\frac{p}{\gamma_k^*}\ln L+\frac{p}{\gamma_k^*}\ln \gamma_k^*+\frac{\gamma_k}{\gamma_k^*}\ln L_k\Big]\\
\\
&=&p^*\ln (L\gamma_k^*)+\frac{p^*}{p}E_k:=r_k+aE_k,
\end{array}
$$
where $r_k=p^*\ln(L\gamma_k^*)$ and $a={p^*}/{p}>1$.

As a consequence of this, we have
\begin{equation}
\label{1.25}E_k\leq r_{k-1}+ar_{k-2}+\dots+a^{k-2}r_1+a^{k-1}E_1.
\end{equation}

Besides this,
$$
\begin{array}{ccl}
 \gamma_k&=&\gamma_{k-1}^*a=(\gamma_{k-1}-1+p)a=\gamma_{k-2}^*a^2+(p-1)a\\
  &=&\gamma_{k-2}a^2+(p-1)a^2+(p-1)a = \dots\dots \\
&=&\gamma_1a^{k-1}+(p-1)a^{k-1}+(p-1)a^{k-2}+\dots+(p-1)a\\
&=&a^{k-1}(p^*-\theta)+\theta,
\end{array}
$$
where $\theta={a(p-1)}/{(1-a)}={p^*(1-p)}/{(p^*-p)}<0$. Hence,
$$
 r_k=p^*\ln (L\gamma_k^*)
=p^*\ln L[a^{k-1}(p^*-\theta)+\theta-1+p] $$ with $\theta-1+p<0.$

So,
$$
r_k\leq p^*\ln [La^{k-1}(p^*-\theta)]=p^*(k-1)\ln a+b,
$$
where $b:=p^*\ln [L(p^*-\theta)]$.

Now, as a consequence of this in (\ref{1.25}), we have
$$
\begin{array}{ccl}
 E_k&\leq&a^{k-1}E_1+\sum_{i=1}^{k-1}a^{i-1}r_{k-i}\\
 \\
&\leq&a^{k-1}E_1+p^*\ln a\sum_{i=1}^{k-1}(k-i-1)a^{i-1}+b\sum_{i=1}^{k-1}a^{i-1}\\
\\
&\leq& a^{k-1}E_1+p^*\ln a \Big(\frac{a^{k-1}-1}{(a-1)^2}\Big) +
b\Big( \frac{a^{k-1}-1}{a-1}\Big),
\end{array}
$$
because we used in last inequality
$$\sum_{i=1}^{k-1}(k-i-1)a^{i-1}\leq\frac{a^{k-1}-1}{(a-1)^2}~~~\mbox{and}~~~\sum_{i=1}^{k-1}a^{i-1}=\frac{a^{k-1}-1}{a-1}.$$

Therefore, we have

$$E_k\leq a^{k-1}E_1+\frac{\{b(a-1)+p^*\ln
a\}(a^{k-1}-1)}{(a-1)^2}.$$ That is,
$$L_k\leq e^{\frac{a^{k-1}E_1+\{b(a-1)+p^*\ln a\}(a^{k-1}-1)/(a-1)^2}{a^{k-1}(p^*-\theta)+\theta}},~\mbox{for each}~k \in \mathbb{N}.$$

Hence,
$$\Vert\omega_\epsilon \Vert_\infty\leq\limsup_{k\to\infty}\Vert\omega_\epsilon\Vert_{\gamma_k}\leq\limsup_{k\to\infty}L_k \leq e^d,$$
where $d=[{E_1+\{b(a-1)+p^*\ln a\}/(a-1)^2}]/[{p^*-\theta}]$ is
bounded above by a constant not depending on $\epsilon$, because
$E_1=\gamma_1\ln L_1$ and $L_1$ is bounded above by a constant
independent of  $\epsilon$. This proves the claim.

As a consequence of this claim, we have by Lieberman \cite{G} that
$\omega_\epsilon\in C^{1,\nu}(\overline{\Omega})$ for some $0<\nu<1$
and $\Vert \omega_\epsilon \Vert_{C^{1,\nu}(\overline{\Omega})} \leq
C$, where $C$ does not depend on $\varepsilon>0$. So, we can define
$$\Lambda^*:=\Lambda^*(\Omega):=\displaystyle
\sup\{(\Vert\nabla\omega_\epsilon\Vert_\infty+1)^{-\alpha}~/~0<\varepsilon<1\}>0.$$

Now, given $0\leq\mu < \Lambda^*$, we have
$$-div((|\nabla\omega_\epsilon|^{p-2}+\epsilon)\nabla\omega_\epsilon)=h(x)\geq\mu h(x)(|\nabla\omega_\epsilon|+1)^\alpha\ \mbox{in}\
\Omega,$$ that is, $\omega_\epsilon$ is a super solution of
(\ref{zepsilon}). Beside this, since $\underline{z}=0 \leq
\omega_\epsilon$ is a sub solution of (\ref{zepsilon}), it follows
by sub and super solution theorem in \cite{K} and regularities
results in \cite{G} the proof of lemma.
\end{proof}
\smallskip

From now on, let us say that $a$ is a \textbf{ $c_{\Omega}$-positive function},  if the following property holds:
$$
 \mbox{If} \,\,  a(x_0)=0~\mbox{for some}~x_0\in \Omega,~\mbox{then there exists }~\Theta\subset \subset \Omega~\mbox{such that}~x_0\in
\Theta~\mbox{and}~a(x)>0~\mbox{on}~ \partial \Theta.
$$

\noindent The below theorem complements the principal results in
Bandle and Giarrusso \cite{BG} by permitting $p
\neq 2$ and non-autonomous potentials $a$ and $b$ and Hamydy
 \cite{H} (and works quoted therein), because it permits $1<p
<\infty$, $\alpha \neq p-1$, non-monotonous term $f$ and more
general terms $a$
\begin{theorem}\label{teo1}
 Suppose $1<p<\infty$, $0\leq\alpha\leq p$, $\liminf_{t\to\infty}{f(t)}/{t^q}>0$ for some $q>\max\{\alpha, p-1, 1\}$, $b\in L^\infty(\Omega)$
and either
$$(a_1)~~a\in C(\Omega)\cap L^\infty(\Omega)~\mbox{is a}~
c_\Omega\!-\!\mbox{positive function}~~~\mbox{or}~~~(a_2)~~ a\in
L^\infty(\Omega)~\mbox{is such
that}~(a_\Omega)^{\prime}~\mbox{holds}.$$ Then,
there exists $0<\mu_*\leq\infty$ such that the problem $(\ref{3.1})$
has at least a solution $u=u_\mu\in C^1(\Omega)$ for each $0\leq\mu
< \mu_*$ given. Besides this, $\mu_*=\infty$, if $(a_2)$ holds.
\end{theorem}

In the proof of the above result, we need of the following technical lemma

\begin{lemma}\label{lema2} Assume  $h:\ [0,\infty)\rightarrow [0, \infty)$ is a continuous function such that
$h(t)>0$ for $t>0$, $h(0)=0$ and
$$\liminf_{s\to\infty}\frac{h(s)}{s^q}>0,\ \mbox{for some}\ q>0.$$
Then there exist increasing functions $\underline{h}, \overline{h}:\
[0,\infty)\rightarrow [0, \infty)$ in $C^1(0,\infty)\cap
C[0,\infty)$ satisfying $\underline{h}(0)=\overline{h}(0)=0$,
$\underline{h}(t) \leq h(t)\leq\overline{h}(t)$, $t>0$,
$$\liminf_{s\to\infty}\frac{\underline{h}(s)}{s^q}>0~~\mbox{and}~~\liminf_{s\to\infty}\frac{\overline{h}(s)}{s^q}>0.$$
\end{lemma}

\noindent \textbf{Proof.} At first, we are going to prove the
existence of $\overline{h}$. Defining $l(t)=\max_{s\in[0,t]}h(s)$,
it is to check that $l$ is continuous and
$$l(t)\geq h(t),~t\geq 0,~~l(0)=0~~\mbox{and}~~l\ \mbox{is
nondecreasing}.$$

To the regularity, we are going to define $\tilde{l}:[0,\infty)\to
[0,\infty)$ by $\tilde{l}(0)=0$ and
$$\tilde{l}(t)=\displaystyle\frac{1}{t}\int_{t}^{2t}{l}(s)ds,\ t>0. $$

So, it is immediate that
$$(\tilde{l})^{'}(t)\geq 0,~~\mbox{and}~~ h(t)\leq l(t)\leq \tilde{l}(t)\leq l(2t),\ \forall\ t\geq0$$
and defining
$$\overline{h}(s)=\tilde{l}(s)+\int_0^s h(\zeta)d\zeta,~s \geq 0,$$
we have the claimed.

Now, let us prove the existence of $\underline{h}$.  Since
$\liminf_{s\to\infty}{h(s)}/{s^q}>0,\ \mbox{for some}\ q>0$, then
there exist positive constants $M$ and $C$ such that $h(s)\geq
Cs^q,s\geq M.$ Set $\eta(s)=\min\{\min_{t \geq s}h(t), Cs^q\}$ for
$s\in [0, M]$, and define
$$ \tilde{h}(t)=\left\{
\begin{array}{cc}
 \displaystyle\frac{1}{M}\int_0^t\eta(s)ds,\ &\ t\in[0,M],\\

\displaystyle\frac{\displaystyle\int_0^M\eta(s)ds}{M^{q+1}}t^q,\ &\
t\in[M,\infty).
\end{array}
\right.
$$
Finaly, defining the $C^1(0,\infty)\cap C[0,\infty)$ function
$\underline{h}:[0,\infty) \to [0,\infty)$ by $\underline{h}(0)=0$
and
$$\underline{h}(t)=\frac{1}{t}\int_{\frac{t}{2}}^t \tilde{h}(s)ds,\ t>0,$$
we have proved the claiming.
\fim

\noindent \textbf{Proof of Theorem \ref{teo1}.} Due to the lack of ellipticity of the
operator $\Delta_p $, we cannot apply standard comparison principle. So, we are
going to consider a modified problem by $0 < \varepsilon <1$ given
by
\begin{equation}\label{epsilon}
 \left\{
\begin{array}{c}
 div((|\nabla u|^{p-2}+\epsilon)\nabla u)=a(x)f(u)+\mu b(x)|\nabla u|^\alpha\ \ \mbox{in}\ \Omega,\\
u\geq0\ \mbox{in}\ \Omega,\ \ u(x)=1\ \mbox{on}\ \partial\Omega.
\end{array}
\right.
\end{equation}

Since $0$ and $1$ are sub and super solutions of (\ref{epsilon})
respectively, it follows by a theorem in Kura [12], that
(\ref{epsilon}) admits a solution $\zeta_1^\epsilon\in
C^{1,\nu}(\overline{\Omega})$ for some $\nu \in (0,1]$, not
depending on $\epsilon$, such that $0\leq \zeta_1^\epsilon\leq 1$ in
$\overline{\Omega}$.

Now, inductively repeating this process, using
$\zeta_{k-1}^\epsilon$ as a sub solution and $k$ as a super
solution, we get a sequence $\zeta_k^\epsilon\in
C^{1,\nu}(\overline{\Omega})$ (the same $\nu$ as before) that
satisfies
\begin{equation}
\label{1.241} 0 \leq  \zeta_{1}^\epsilon \leq  \zeta_{2}^\epsilon
\leq \cdots \leq \zeta_{k-1}^\epsilon\leq\zeta_k^\epsilon\leq
k~~\mbox{in}~~\overline{\Omega}
\end{equation}
 and
\begin{equation}\label{epsilonk}
 \left\{
\begin{array}{c}
 div((|\nabla u|^{p-2}+\epsilon)\nabla u)=a(x)f(u)+\mu b(x)|\nabla u|^\alpha\ \ \mbox{in}\ \Omega,\\
u\geq0\ \mbox{in}\ \Omega,\ \ u(x)=k\ \mbox{on}\ \partial\Omega.
\end{array}
\right.
\end{equation}

As a consequence of this and Lemma \ref{lema2} with $h=f$, we have
$\zeta_k^\epsilon\in C^{1,\nu}(\overline{\Omega})$ satisfies
\begin{equation}\label{epsilon0}
 \left\{
\begin{array}{c}
 div((|\nabla u|^{p-2}+\epsilon)\nabla u)\geq a(x)\underline{f}(u)-\mu b^-(x)(|\nabla u|+1)^\alpha\ \ \mbox{in}\ \Omega,\\
u\geq0\ \mbox{in}\ \Omega,\ \ u(x)=k\ \mbox{on}\ \partial\Omega.
\end{array}
\right.
\end{equation}

\noindent Now, we are going to assume  $(a_1)$.
\medskip

\noindent\textbf{Claim:} {\em For each $x\in\Omega$, there exist a
open $V_x \subset \subset\Omega$ and a function  $\zeta_x\in
C^{2}(V_x)$ satisfying
$$0 \leq  \zeta_{1}^\epsilon \leq  \zeta_{2}^\epsilon
\leq \cdots \leq \zeta_{k-1}^\epsilon\leq\zeta_k^\epsilon\leq\cdots
\leq \zeta_x~\mbox{in}~V_{x},~\mbox{for all}~0<\varepsilon <1,~k\in
\mathbb{N}$$ and $0 \leq \mu < \Lambda^*(\Omega)$ given, where
$\Lambda^*(\Omega)> 0$ was defined in Lemma $\ref{lema31}$.}

In fact, given a $x_0\in\Omega$, we are going to consider two cases:\\
\noindent$\underline{Case~ 1:}$ $a(x_0)>0$. In this case, consider
$V_{x_0}\subset\Omega$ a smooth open domain such that $a(x) \geq
a_0>0$ for all $x\in \overline{V}_{x_0}$, $v\in
C^2(\overline{V}_{x_0})$ the solution of problem
 \begin{equation}
 \label{1.27}
 \left\{
\begin{array}{c}
-\Delta u=1\ \ \mbox{in}\ V_{x_0},\\
u>0\ \mbox{in}\ V_{x_0},\ \ u(x)=0\ \mbox{on}\ \partial V_{x_0}
\end{array}
\right.
\end{equation}
and denote by $g(x)=-\Delta_p v(x),~x \in V_{x_0}$. So, $g\in
L^\infty(V_{x_0})$.

Besides this, by Lemma \ref{lema2}, there exist  $s_0>0$ such that
$\underline{f}(s)\geq cs^q$ for $s\geq s_0$, where
$d=\liminf_{s\to\infty}{\underline{f}(s)}/{s^q}>0 \,\,\, \mbox{for
some} \,\, q>\max\{\alpha, p-1,1\}$ and $c=d/2$. Now, defining
$\omega=Mv^{-\beta}\in C^2(V_{x_0})$, where $M$ and $\beta$ are
positive real parameters, we have, for each $0\leq\varphi\in
C_0^\infty(V_{x_0})$, that
\begin{eqnarray*}
\begin{array}{l}
  \displaystyle\int_{V_{x_0}}(|\nabla\omega|^{p-2}+\epsilon)\nabla\omega\nabla\varphi dx+\int_{V_{x_0}}[ca(x)\omega^q-\mu b^-(x)
(|\nabla\omega|+1)^\alpha]\varphi dx= \\
\\
 \displaystyle -\int_{V_{x_0}}\beta^{p-1}M^{p-1}v^{(-\beta-1)(p-1)}|\nabla v|^{p-2}\nabla v\nabla\varphi dx-\int_{V_{x_0}}\epsilon\beta Mv^{-\beta-1}
\nabla v\nabla\varphi dx+ \\
\\
  \displaystyle\int_{V_{x_0}}[ca(x)M^qv^{-\beta q}-\mu b^-(x)(\beta  Mv^{-\beta-1}|\nabla v|+1)^\alpha]\varphi
  dx=\\
  \\
-\displaystyle\int_{V_{x_0}}\beta^{p-1}M^{p-1}|\nabla v|^{p-2}\nabla
v\nabla(v^{(-\beta-1)(p-1)}\varphi)dx-\int_{V_{x_0}}\beta^{p-1}M^{p-1}
(\beta+1)(p-1)v^{(-\beta-1)(p-1)-1}|\nabla v|^p\varphi dx\\
\\
-\displaystyle\int_{V_{x_0}}\epsilon\beta M\nabla v\nabla
(v^{-\beta-1}\varphi)dx-\int_{V_{x_0}}\epsilon\beta
M(\beta+1)v^{-\beta-2}|\nabla v|^2 \varphi dx+\\
\\
\displaystyle\int_{V_{x_0}}[ca(x)M^qv^{-\beta q}-\mu b^-(x)(\beta
Mv^{-\beta-1}|\nabla v|+1)^\alpha]\varphi dx.
\end{array}
\end{eqnarray*}

So, from (\ref{1.27}) and $g=-\Delta_p v$, we get
\begin{eqnarray*}
\begin{array}{l}
  \displaystyle\int_{V_{x_0}}(|\nabla\omega|^{p-2}+\epsilon)\nabla\omega\nabla\varphi dx+\int_{V_{x_0}}[ca(x)\omega^q-\mu b^-(x)
(|\nabla\omega|+1)^\alpha]\varphi dx\geq \\
\\
-\displaystyle\int_{V_{x_0}}\beta^{p-1}M^{p-1}g(x)v^{(-\beta-1)(p-1)}\varphi
dx-\int_{V_{x_0}}\beta^{p-1}M^{p-1}
(\beta+1)(p-1)v^{(-\beta-1)(p-1)-1}|\nabla v|^p\varphi dx\\
\\
-\displaystyle\int_{V_{x_0}}\epsilon\beta Mv^{-\beta-1}\varphi
dx-\int_{V_{x_0}}\epsilon\beta M(\beta+1)v^{-\beta-2}|\nabla v|^2
\varphi dx
\\
\\
+\displaystyle\int_{V_{x_0}}[ca(x)M^qv^{-\beta q}-\mu
b^-(x)2^\alpha(\beta^\alpha M^\alpha v^{\alpha(-\beta-1)}|\nabla
v|^\alpha+1)]\varphi dx.
\end{array}
\end{eqnarray*}

Now, fixing
$$\beta=\displaystyle\max\Big\{\frac{\alpha}{q-\alpha},\frac{p}{q-p+1},\frac{2}{q-1}\Big\},$$
we have
$$\min\{(-\beta-1)(p-1)-1+\beta q, (-\beta-1)\alpha+\beta q, (-\beta-2)+\beta q\}\geq0.$$
and as a consequence of this and $0\leq \epsilon <1$,  we have
\begin{eqnarray*}
\begin{array}{l}
\displaystyle\int_{V_{x_0}}(|\nabla\omega|^{p-2}+\epsilon)\nabla\omega\nabla\varphi
dx+\int_{V_{x_0}}[ca(x)\omega^q-\mu b^-(x)(|\nabla\omega|+1)^\alpha]\varphi dx\geq \\
\\
\displaystyle\int_{V_{x_0}}M^{p-1}v^{-\beta q}\Big[-\beta^{p-1}\Vert
g \Vert_{\infty}\Vert v \Vert_{\infty}^{(-\beta-1)(p-1)+\beta
q}-\beta^{p-1}
(\beta+1)(p-1)\Vert v\Vert_{\infty}^{(-\beta-1)(p-1)-1+\beta q}\Vert\nabla v\Vert_{\infty}^p \\
\\
-\displaystyle\beta M^{2-p}\Vert v\Vert_{\infty}^{-\beta-1+\beta
q}-\beta
M^{2-p}(\beta+1)\Vert v\Vert_{\infty}^{-\beta-2+\beta q}\Vert\nabla v\Vert_{\infty}^2\\
\\
-\displaystyle\mu \Vert b\Vert_{\infty}2^\alpha\Big(\beta^\alpha
M^{\alpha-p+1} \Vert v\Vert_{\infty}^{\alpha(-\beta-1)+\beta
q}\Vert\nabla v\Vert_{\infty}^\alpha+M^{1-p}\Vert
v\Vert_{\infty}^{\beta q}\Big)+cM^{q-p+1} a_0\Big]\varphi dx.
\end{array}
\end{eqnarray*}

Now, since $q>\max\{\alpha, p-1,1\}$, we can choose a constant
$M=M_{\mu,V_{x_{0}}}>0$ (not depending on $\epsilon$) large enough
such that
$$\int_{V_{x_0}}(|\nabla\omega|^{p-2}+\epsilon)\nabla\omega\nabla\varphi dx+\int_{V_{x_0}}[ca(x)\omega^q-\mu b^-(x)
(|\nabla\omega|+1)^\alpha]\varphi dx\geq0$$ and defining
$\zeta_{x_0}(x)=\omega(x)+s_0$ (not depending on $\epsilon$), we
have that $\zeta_{x_0} \in C^2(V_{x_0})$ and satisfies
\begin{equation}\label{eqlimitado}
\left\{
\begin{array}{c}
div((|\nabla u|^{p-2}+\epsilon)\nabla u)\leq
a(x)\underline{f}(u)- \mu b^-(x)(|\nabla u|+1)^\alpha\ \mbox{in}\ V_{x_0},\\
u\geq s_0\ \ \ \mbox{in}\ V_{x_0},\ u(x)\stackrel{d(x)\to
0}{\longrightarrow}\infty
\end{array}
\right.
\end{equation}
for each $0 \leq \epsilon < 1$ and $\mu \geq 0$
given.

Besides this, for $0 < \epsilon < 1$ (that is,
$\epsilon \neq 0$) given, it follows from (\ref{epsilon0}),
(\ref{eqlimitado}), and a comparison principle in \cite{GT}, that
$$0 \leq \zeta_k^\epsilon\leq\zeta_{x_0}~\mbox{in}~V_{x_0},~\mbox{for all}~k\in\mathbb{N}.$$

\noindent$\underline{Case~2:}$ $a(x_0)=0$.  Since $a$ is a
$c_\Omega-$positive function, there exists a open
$V_{x_0}\subset\Omega$ such that
$$
x_0\in V_{x_0} \,\,\, \mbox{and} \,\,\,  a(x)>0 \,\,\, \mbox{for
all} \,\,\, x\in\partial V_{x_0}.
$$
Taking a finite cover of $\partial V_{x_0}$, namely  $V_i,\
i=1,\dots,n$, such that
$$
\partial V_{x_0}\subset \bigcup_{i=1}^nV_i \,\,\, \mbox{and} \,\,\, a(x) \geq a_i>0,\ x\in V_i,
$$
it follows from the argument of the case 1 that there exists
$\zeta_{x_0}^i\in C^{2}(V_i)$ such that $0 \leq \zeta_k^\epsilon\leq
\zeta_{x_0}^i$ in $V_i$ for all $0<\epsilon<1$ and $k \in
\mathbb{N}$. In particular, there exists a positive real constant
$A=A_{x_0}>0$ such that $\zeta_k^\epsilon\leq A$ on $\partial
V_{x_0}$, $\forall\ k\in\mathbb{N}$ and $0<\varepsilon < 1$.

Now, taking $u=u_{\epsilon,\mu}\in C^1(\overline{\Omega})$ for
$0\leq \mu< \Lambda^{*}$ a solution of problem (\ref{zepsilon}),
given by Lemma \ref{lema31}, we have that $A+u_{\epsilon,\mu}$
satisfies
$$
 \left\{
\begin{array}{c}
div((|\nabla u|^{p-2}+\epsilon)\nabla u)\leq a(x)\underline{f}(u)-\mu b^-(x)(|\nabla u|+1)^\alpha\ \ \mbox{in}\ V_{x_0},\\
u\geq A\ \mbox{in}\ V_{x_0}\ \ u(x)\geq A\ \mbox{on}\ \partial
V_{x_0}
\end{array}
\right.
$$
and $\zeta^{\epsilon}_k\leq A\leq A+u_{\epsilon,\mu}$ on $\partial
V_{x_0}$. So, it follows of a comparison principle in \cite{GT} that
$\zeta^{\epsilon}_k\leq A+u_{\epsilon,\mu}$ in $V_{x_0}$. Since by
Lemma 3.1, we have $\Vert u_{\epsilon,\mu}\Vert_{\infty} \leq  C$,
with $C>0$ not depending on $\epsilon$, the claim follows by taking
$\zeta_{x_0}=A+C$.

As a consequence of the both prior cases, it follows  that given a
compact set $K\subset\Omega$ there exists a constant $C_K>0$ such
that
\begin{equation}
\label{diag}0\leq\zeta_1^\epsilon\leq\zeta_2^\epsilon\leq\cdots\leq\zeta_k^\epsilon\leq\cdots\leq
C_K~\mbox{in}~K\ \mbox{and}\ \zeta_k^\epsilon\in C^{1,\nu}
(\overline{K})\ \mbox{for all}~\epsilon \in(0,1)~\mbox{and}~k \in
\mathbb{N}.
\end{equation}

That is, taking $\epsilon_n\in (0,1)$ with
$\epsilon_n\to 0$ and $\Omega_j\subset \subset\Omega$ smooth open
sets such that
\begin{equation}
\label{C} \Omega_j\subset
\subset\Omega_{j+1}~\mbox{and}~\Omega=\cup_{j=1}^\infty\Omega_j,
\end{equation}
 it follows from
(\ref{diag}), that there exist subsequences of $(\epsilon_n)$,
denoted by $(\epsilon_{n_{ji}})$, where
$$\cdots
\subseteq N_j \subseteq N_{j-1}\subseteq \cdots \subseteq N_1
\subseteq N~\mbox{with}~ N_j=\{n_{j1},n_{j2},n_{j3}, \cdots\},$$
such that $\zeta_{k}^{\epsilon_{n_{ji}}}\stackrel{i \to
\infty}{\longrightarrow} \zeta_{k}^j~ \mbox{in}~
C^{1,\theta}(\overline{\Omega}_j)$ for some $0<\theta<\nu \leq 1$,
with $\theta$ does not depend on $\epsilon$, and
${\zeta_{k}^j}_{\mid_{\overline{\Omega}_{j-1}}}=\zeta_{k}^{j-1}$ for
each $k,j \in \mathbb{N}$.

Now, defining $\zeta_k=\zeta_{k}^j$ for $x\in\overline{\Omega}_j$,
it follows that $\zeta_{k}^{\epsilon_{n_{jj}}}\stackrel{j \to
\infty}{\longrightarrow} \zeta_k$ in
$C^{1,\vartheta}_{loc}({\Omega})$ for some $0<\vartheta<\theta<1$,
with $\vartheta$ does not depending on $\epsilon$ with $\zeta_k$
satisfying
\begin{equation}
\label{A} 0\leq\zeta_1\leq\zeta_2\leq\dots\leq\zeta_k\leq \cdots
\leq C_{\overline{\Omega}_j}~ \mbox{in}~
\overline{\Omega}_j~\mbox{for each}~ j\in\mathbb{N}\end{equation}
and
\begin{equation}
\label{B} \left\{
\begin{array}{c}
\Delta_pu=a(x)f(u)+\mu b(x)|\nabla{u}|^\alpha\ \ \ \mbox{in}\ \Omega,\\
u\geq0\ \ \ \mbox{in}\ \Omega,\ u(x)=k~\mbox{on}\ \partial\Omega
\end{array}
\right.
\end{equation}
for each $k \in \mathbb{N}$ given.

 Hence, applying the diagonal process again, now in $k$, it follows from (\ref{A}) and (\ref{B})
 that there exists a $\zeta \in C^1(\Omega)$ solution of
(\ref{3.1}).
\medskip

\noindent Now, we are going to assume  $(a_2)$.
\medskip

In what follows, we will take $\Omega_j\subset \subset\Omega$ smooth
open sets satisfying (\ref{C}) again. Then, it follows from
hypothesis $(a_\Omega)^{\prime}$ that there exists $a_{\Omega_n}>0$
such that $a(x)\geq a_{\Omega_n}$ in $\Omega_n$. This permit us, in
a similar way to Case 1, to build a function $\overline\omega_n \in
C^2(\Omega_n)$ ($\overline\omega_n$ independent of $\varepsilon$)
satisfying
 \begin{eqnarray}
\label{4.7.3} \left\{
\begin{array}{l}
div((|\nabla \omega|^{p-2}+\epsilon)\nabla \omega)\leq
a(x)\underline{f}(\omega)-\mu b^-(x)(|\nabla \omega|+1)^\alpha~
\mbox{in}\ \Omega_n,\\

\omega\geq0\ \mbox{in}\
\Omega_n,~~\omega(x)\stackrel{d(x)\to0}{\longrightarrow}\infty
\end{array}
\right.
\end{eqnarray}
for each $0 \leq \epsilon <1$ and $\mu \geq 0$ given.

Beside this, for each $0 < \epsilon <1$, we have  $0
\leq \zeta_k^\epsilon\leq\overline{\omega}_n$ in $\Omega_n$ for all
$k\in\mathbb{N}$, where $\zeta_k^\epsilon \in
C^{1,\nu}(\overline{\Omega})$ satisfies (\ref{1.241}) and
(\ref{epsilonk}). So, given  a compact set $K\subset\Omega$ there
exists a $n_K \in \mathbb{N}$ such that $K \subset \Omega_{n_K}$.
Thus, there exists a constant $C_K>0$ such that (\ref{diag}) holds
again.

That is, under the notations of last diagonal process, we obtain
$\zeta_{k}^{\epsilon_{n_{ji}}}\stackrel{i \to
\infty}{\longrightarrow} \zeta_{k}^j~ \mbox{in}~
C^{1,\theta}(\overline{\Omega}_j)$ for some $0<\theta<\nu \leq 1$,
with $\theta$ does not depend on $\epsilon$, $0 \leq\zeta_{k}^j \leq
\overline{\omega}_{j+1}$ in $\overline{\Omega}_j$ and
${\zeta_{k}^j}_{\mid_{\overline{\Omega}_{j-1}}}=\zeta_{k}^{j-1}$ for
each $k,j \in \mathbb{N}$. So, repeating the argument as before, we
get a  that is a solution of (\ref{3.1}). These end the proof of
Theorem 3.1. \fim

\section{\textbf{Proof of Theorem \ref{teo2}}}

First, we are going to consider the case ($P$)$_{\rho}$--$(ii)$,
because in the proof of ($P$)$_{\rho}$--$(i)$ we let us use the
proof of the first case with $\mu=0$.
\medskip

\noindent \textbf{Case 1}: Assume ($P$)$_{\rho}$--$(ii)$, that is, $b^+\neq 0$. \\
\medskip

At first, we are going to build a nonnegative sub solution
$\underline{u}$ of (\ref{eq1}) by proving the existence of a
solution for the problem
\begin{equation}\label{14}
\left\{
\begin{array}{c}
\Delta_pu=a(x)\overline{f}(u)+\mu b(x)|\nabla u|^{p-1}\ \mbox{in}\ \mathbb{R}^N,\\
u\geq0\ \mbox{in}\ \mathbb{R}^N,~~ u(x)\stackrel{|x|\to
+\infty}{\longrightarrow}+\infty,
\end{array}
\right.
\end{equation}
where $\overline{f}$ was built as in  Lemma \ref{lema2}.

To do this, first we note that of Theorem \ref{teo1}, we get a
$\underline{u}_n \in C^1(B_n)$  solution of  problem
$$
\left\{
\begin{array}{c}
\Delta_pu=a(x)\overline{f}(u)+\mu b(x)|\nabla u|^{p-1}\ \mbox{in}\ B_n,\\
u\geq0\ \mbox{in}\ B_n,~~ u(x)=+\infty,\ \mbox{on}\ \partial B_n
\end{array}
\right.
$$
and as a consequence of Theorem 2.1, we have
$\underline{u}_n\geq\underline{u}_{n+1}\geq0\ \mbox{in}\ B_n$. In
this case, $\mu_*=\mu_*(B_n)=\infty$, since ($a_2$) holds for each
$n \in \mathbb{N}$.

So, by a diagonal process, we can show that
$\underline{u}_n\longrightarrow\underline{u}$ in $C^1(\mathbb{R}^N)$ that satisfies
$$\int_{\mathbb{R}^N}|\nabla \underline{u}|^{p-2}\nabla \underline{u}\nabla \phi dx+\int_{\mathbb{R}^N}[a(x)\overline{f}(\underline{u})+
\mu b(x)|\nabla \underline{u}|^{p-1}]\phi dx=0, \ \phi\in
C_0^\infty(\mathbb{R}^N).$$  To complete the building of
$\underline{u}$, just remain to prove that $\underline{u}(x) \to
+\infty$ when $|x| \to +\infty$. To do this, defining $\omega^n \in
C^{1}({B}_n)$ by
\begin{equation}\label{omeganm}
\omega^n(x)=\int_{\underline{u}_n(x)}^\infty
\big(\overline{f}(t)+1\big)^{-\frac{1}{p-1}}dt,~x\in B_n,
\end{equation}
 we have $\omega^n>0$ in $B_n$,
$\omega^n(x)=0$ on $\partial B_n$ and
$$
\label{4.6}
\begin{array}{cll}\displaystyle\int_{B_n}|\nabla\omega^n|^{p-2}\nabla\omega^n\nabla\varphi dx
&=&-\displaystyle\int_{B_n}\overline{f}((\underline{u}_n)+1)^{-1}|\nabla \underline{u}_n|^{p-2}\nabla \underline{u}_n\nabla\varphi dx\\
&\leq&\displaystyle\int_{B_n}\overline{f}((\underline{u}_n)+1)^{-1}[a(x)\overline{f}(\underline{u}_n)+\mu b(x)|\nabla \underline{u}_n|^{p-1}]\varphi dx\\
&\leq& \displaystyle \int_{B_n} [a(x)+\mu
b^+(x)|\nabla\omega^n|^{p-1}]\varphi dx.
\end{array}
$$

That is,
$$\displaystyle\int_{B_n}|\nabla\omega^n|^{p-2}\nabla\omega^n\nabla\varphi dx\leq \displaystyle \int_{B_n} [a(x)(1 +(\omega+1)^\eta/2)+\mu
b^+(x)|\nabla\omega^n|^{p-1}]\varphi dx,$$ for every $\varphi \in
C^{\infty}(\mathbb{R}^N)$ with $\varphi \geq 0$.

 So, given $0 \leq \mu < \Lambda_*$, it follows from
Theorem 2.1 that $\omega^n\leq \omega_\mu$ in $B_n$ for all $n$,
where $\Lambda_*$ and $\omega_\mu$ were given in Lemma \ref{lema1}.
Since $\underline{u}_n\to \underline{u}$ in $C^1(\mathbb{R}^N)$, it
follows from (\ref{omeganm}) that there exists a $\omega_0\in
C^1(\mathbb{R}^N)$ with $\omega_0\leq\omega_\mu$ and $\omega_0(x)\to
0$ as ${|x|\to\infty}$ such that $\omega^n\to\omega_0$ in
$C^1(\mathbb{R}^N)$ and
$$\omega_0(x)=\int^\infty_{\underline{u}(x)}\big(\overline{f}(t)+1\big)^{-\frac{1}{p-1}}dt,~x \in \mathbb{R}^N.$$
 As a consequence of this, we have $\underline{u}(x)\to\infty$ as
 ${|x|\to\infty}$. This shows that $\underline{u}$ is a solution of (\ref{14}), that is, $\underline{u}$ is a sub solution
of (\ref{eq1}).

 Now, considering the problem
\begin{eqnarray}
\label{bnn}\left\{
\begin{array}{c}
 \Delta_p u=a(x)f(u)+\mu b(x)|\nabla u|^{p-1}\ \ \mbox{in}\ \ B_n,\\
u\geq 0\ \mbox{in}\ B_n,~~ u(x)=\underline{u}(x), \ \ \mbox{on}\
\partial B_n
\end{array}
\right.
\end{eqnarray}
we have that $\underline{u}$ and $\overline{\omega}_n$ are sub and
super of (\ref{bnn}) and $\underline{u}\leq\overline{\omega}_n$ em
$B_n$, where $\overline{\omega}_n$ satisfies (\ref{4.7.3}) with
$\Omega_n=B_n$ and $\epsilon=0$. Then, by sub and super solution
method and regularity theory, the problem (\ref{bnn}) has a solution
$u_n\in C^1(B_n)$ with $\underline{u}\leq
u_n\leq\overline{\omega}_n$ for all $n \in \mathbb{N}$.

So, applying the Theorem 2.1 again, we have $\underline{u}\leq
u_m\leq \overline{\omega}_n$ in $B_n$ for all $m,n \in \mathbb{N}$
such that $m\geq n$ and as a consequence of this, by a diagonal
process, there is a function
 $u\in C^1(\mathbb{R}^N)$ and a subsequence of ${u_n}$, denoted by itself, such that $u_n\to u$ with  $u\geq \underline{u}$
 in $\mathbb{R}^N$ and $u$ a solution of (\ref{eq1}).
\medskip

\noindent \textbf{Case 2}: Suppose ($P$)$_{\rho}$--$(i)$. \\

At first, given $\epsilon \in (0,1)$ and $n \in \mathbb{N}$, we are
going to consider $\zeta_1^{\epsilon,n}\in C^1(\overline{B}_n)$ and
$\overline{\omega}^n \in C^1({B_n})$ solutions of the problems
(\ref{epsilon}) and (\ref{4.7.3}), respectively in $B_n$. So,
$\zeta_1^{\epsilon,n}$ and $\overline{\omega}^n$ are sub and super
solutions of the problem
\begin{equation}\label{bmenos}
 \left\{
\begin{array}{c}
div((|\nabla u|^{p-2}+\epsilon)\nabla u)=a(x)\underline{f}(u)-\mu b^-(x)(|\nabla u|+1)^\alpha\ \  \mbox{in}\ B_n,\\
u\geq0\ \mbox{in}\ B_n,\ u(x)=1\ \mbox{on}\ \partial B_n
\end{array}
\right.
\end{equation}
and, by standard principle comparison, we have
$\zeta_1^{\epsilon,n}\leq\overline{\omega}^n$ in $ B_n$. We remember
that $\overline{\omega}^n$ does not depend of $\epsilon \in (0,1)$.

{Now, taking $B_{n-{1}/{k}}\subset B_n$, where $k\in\mathbb{N}$,} it
follows by a sub and super solution of \cite{K} and a result of
regularities in \cite{G} that the problem
$$
 \left\{
\begin{array}{c}
div((|\nabla u|^{p-2}+\epsilon)\nabla u)=a(x)\underline{f}(u)-\mu b^-(x)(|\nabla u|+1)^\alpha\ \  \mbox{in}\ B_{n-{1}/{k}},\\
u\geq0\ \mbox{in}\ B_{n-{1}/{k}},\
u(x)=\zeta_1^{\epsilon,n}|_{B_{n-{1}/{k}}}\ \mbox{on}\ \partial
B_{n-{1}/{k}}
\end{array}
\right.
$$
admits a solution $u_{1,k}^{\epsilon,n}\in
C^{1,\nu}(\overline{B_{n-{1}/{k}}})$, for some $0<\nu \leq 1$. After
this, {applying a diagonal process in $k$}, we show that that  the
problem (\ref{bmenos}) admits a solution $u_1^{\epsilon,n}\in
C^{1,\theta}(\overline{B}_n)$, for some $0<\theta<\nu$, such that
$\zeta_1^{\epsilon,n}\leq u_1^{\epsilon,n} \leq \overline{\omega}^n$
in $ B_n$.

Repeating this process, by using $u_{k-1}^{\epsilon,n}$ as a sub
solution and $\overline{\omega}^n$ as a super solution, we obtain a
sequence $\{u_k^{\epsilon,n}\}_{k=1}^{\infty}\in
C^{1,\theta}(\overline{B}_n)$ satisfying
\begin{equation}
\label{137} 0\leq u_1^{\epsilon,n}\leq u_2^{\epsilon,n}\leq\dots\leq
u_{k-1}^{\epsilon,n}\leq u_k^{\epsilon,n}\leq \dots
\leq\overline{\omega}^n~\mbox{in}~\overline{B}_n
\end{equation}
 and
$$
\left\{
\begin{array}{c}
div((|\nabla u_k^{\epsilon,n}|^{p-2}+\epsilon)\nabla u_k^{\epsilon,n})=a(x)\underline{f}(u_k^{\epsilon,n})-\mu b^-(x)
(|\nabla u_k^{\epsilon,n}|+1)^\alpha\ \  \mbox{in}\ B_n,\\
u_k^{\epsilon,n}\geq0\ \mbox{in}\ B_n,\ u_k^{\epsilon,n}(x)=k\
\mbox{on}\ \partial B_n.
\end{array}
\right.
$$

Now, by a diagonal process, it follows from (\ref{137}) that there
exists a function $u^{\epsilon,n}\in C^{1,\vartheta}(B_n)$, for some
$0<\vartheta<\theta$, such that
$$
\left\{
\begin{array}{c}
div((|\nabla u^{\epsilon,n}|^{p-2}+\epsilon)\nabla u^{\epsilon,n})=a(x)\underline{f}(u^{\epsilon,n})-\mu b^-(x)(|\nabla u^{\epsilon,n}|
+1)^\alpha\ \  \mbox{in}\ B_n,\\
u^{\epsilon,n}\geq0\ \mbox{in}\ B_n,\ u^{\epsilon,n}(x)=\infty\
\mbox{on}\
\partial B_n
\end{array}
\right.
$$
and, by comparison principle in \cite{GT},
$$
\label{138} 0 \leq u^{\epsilon,n+1}\leq
u^{\epsilon,n}\leq\overline{\omega}^n\ \mbox{in}\ B_n.
$$

So, following the same argument as in the proof of Case 2 of Theorem
3.1, we show that there exists a $u^n\in C^1(B_n)$ solution of
 the problem
$$
 \left\{
\begin{array}{c}
\Delta_pu=a(x)\underline{f}(u)-\mu b^-(x)(|\nabla u|+1)^\alpha\ \  \mbox{in}\ B_n,\\
u\geq0\ \mbox{in}\ B_n,\ u(x)\stackrel{x\to\partial
B_n}{\longrightarrow}\infty
\end{array}
\right.
$$
satisfying
$$
\label{1371} 0\leq \cdots \leq  u^{n+1}\leq u^{n}\leq
\overline{\omega}^n~\mbox{in}~\overline{B}_n.
$$

On the other hand, it follows from the case 1, with $\mu=0$, (In
this case, in the proof of Lemma 2.2, it is necessary just that the
solution of (P$_{\rho}$) belongs to $C^1(\mathbb{R}^N)$) that there
exists a $v\in C^1(\mathbb{R^N})$ satisfying
$$
 \left\{
\begin{array}{c}
\Delta_pv=a(x)f(v)\geq a(x)\underline{f}(v)\ \  \mbox{in}\ \mathbb{R}^N,\\
v\geq0\ \mbox{in}\ \mathbb{R}^N,\
v(x)\stackrel{|x|\to\infty}{\longrightarrow}\infty.
\end{array}
\right.
$$
 Beside this,  by
comparison principle \cite{PT}, we have $v\leq u^n$ in $B_n$ for all
$n\in~\mathbb{N}$.

So, by a diagonal process, there exists a $\overline{u}\in
C^1(\mathbb{R}^N)$ such that $v\leq \overline{u}$ in $\mathbb{R}^N$,
$u^n\to \overline{u}$ in $C^1(\mathbb{R}^N)$ and $\overline{u}$ is a
solution of the problem
$$
 \left\{
\begin{array}{c}
\Delta_pu=a(x)\underline{f}(u)-\mu b^-(x)(|\nabla u|+1)^\alpha\ \  \mbox{in}\ \mathbb{R}^N,\\
u\geq0\ \mbox{in}\ \mathbb{R}^N,\
u(x)\stackrel{|x|\to\infty}{\longrightarrow}\infty.
\end{array}
\right.
$$

Thus, since $v$ and $\overline{u}$ are sub and super solutions of
the problem
\begin{equation}\label{bmenosprincipal}
 \left\{
\begin{array}{c}
\Delta_pu=a(x)f(u)-\mu b^-(x)|\nabla u|^\alpha\ \  \mbox{in}\ \mathbb{R}^N,\\
u\geq0\ \mbox{in}\ \mathbb{R}^N,\
u(x)\stackrel{|x|\to\infty}{\longrightarrow}\infty
\end{array}
\right.
\end{equation}
it follows by a theorem of sub e super solution in \cite{K}, that
there exists a solution $u\in C^1(\mathbb{R}^N)$ for the  problem
(\ref{bmenosprincipal}). This finishes the proof. \fim

As an immediate consequence of the arguments used in the proof of
last theorem, we have

\begin{corollary} \label{T3}
Assume that $\Omega\subset \mathbb{R}^N$ is smooth bounded domain,
$\liminf_{t\to\infty}{f(t)}/{t^q}>0$ for some $q>\max\{\alpha, p-1,
1\}$ and $a, b\in L^\infty_{loc}(\Omega)$ with $a$ satisfying
$(a_\Omega)^{\prime}$ and
\begin{equation}\label{extra}
 \left\{
\begin{array}{c}
-\Delta_pw=\rho(x)\ \  \mbox{in}\ \Omega,\\
w>0\ \mbox{on}\ \Omega,\ w(x)=0\ \partial\Omega,
\end{array}
\right.
\end{equation}
has a solution in $C^1(\overline{\Omega})$, where
$\rho(x)=\max\{a(x), b^+(x)\}$, $x \in \Omega$ with either
$$(i)~~b^+=0~\mbox{and}~0 \leq \alpha \leq p~~~~\mbox{or}~~~~(ii)~~b^+ \neq 0~\mbox{and}~\alpha = p-1.$$ Then
there exists $\mu^{\star} \in (0, +\infty]$ such that the problem
\begin{equation}
\label{1.36} \left\{
\begin{array}{c}
\Delta_pu=a(x)f(u)+\mu b(x)|\nabla{u}|^\alpha\  \mbox{in}~\Omega,\\
u\geq0\ \mbox{on}\ \Omega,\ u(x)\stackrel{d(x)\to
0}{\longrightarrow}\infty
\end{array}
\right.
\end{equation}
 has a solution in $C^1(\Omega)$, for each $0 \leq \mu <
\mu^{\star}$ given. In additional, if $(i)$ holds, then
$\mu^{\star}=+\infty$.
\end{corollary}

This Corollary complements some above quoted results
principally by permitting the oscillatory and explosive behavior of
potentials $a$ and $b$ on boundary of $\Omega$. In particular, it
complements a result by  Liu e Yang \cite{LY} that considered in
(\ref{1.36}) the nonlinearity $f$ as a non-decreasing function
satisfying $f(s)\leq C_1s^{p_1(p-1)}$ for $s\in (0,\infty)$,
$f(s)\geq C_2s^{p_2(p-1)}$ for $s>>0$, where $p_1\geq p_2$,
$b(x)=\pm 1$ and $a$ satisfying $C_3(d(x))^{\gamma_2}\leq a(x)\leq
C_4 (d(x))^{\gamma_1}$ for all $x\in\Omega$ with $-p<\gamma_1\leq
\gamma_2$ and $C_i$ positive constants.

As examples of non-null and non-negative potentials $\rho$ satisfying  $(\ref{extra})$, we have:
\begin{enumerate}
\item [$(i)$] { $a, b^+ \in L^q(\Omega)$ for some $q>N>1$. For details, see \cite{P},}
\end{enumerate}
\begin{enumerate}
\item [$(ii)$] { $a,b \in C(\Omega)$ such that $a(x),b^{+}(x) \leq C_0d(x)^{-\gamma(x)},\ x\in\Omega,$
 where  $\gamma\in C(\overline{\Omega})$   and $\gamma(x)<{1}/{N}$ for $x\in\partial\Omega$, for some positive constant $C_0$.
 This situation permits singular behaviors for the potential $a$ in the sense that $a(x)\stackrel{x\to
x_0}{\longrightarrow}\infty$ and $a(x)\stackrel{x\to
x_1}{\longrightarrow}a_o<\infty$ for $x_0 \neq x_1$. The same can
occur for $b$ too. For more details, see \cite{J}.}
\end{enumerate}

\section{Proof of Theorem \ref{teo4}}

The proof of Theorem \ref{teo4} consists principally of delicate and
sensible estimates involving the operator and the nonlinearities. In
this result, we are mainly interested in showing nonexistence of
entire solutions that blow-up at infinity. In the literature there
are some results that prove nonexistence of either subsolutions,
supersolutions or solutions without requiring their behavior at
infinity and demanding strongest conditions under the
nonlinearities.
\smallskip

\noindent\proof. {Given $R>0$ define $\xi_R\in C^1(\mathbb{R}^N,
\mathbb{R})$ such that $\xi_R(x)= 1,~ 0\leq|x|\leq R$ and $\xi_R(x)=
0,~ |x|\geq 2R$  satisfying
$$0\leq\xi_R(x)\leq 1,\ \ |\nabla\xi_R(x)|\leq \frac{1}{R}, ~x \in \mathbb{R}^N.$$

Now, considering the $C^1$-functions $\chi=\xi_R^\mu$ and $u^{\beta}
\chi$, where $\mu,\beta>1$ are real parameters, and using the last
one as a test function in (\ref{non}), we get
$$\int_{\mathbb{R}^N}a(x)f(u)u^\beta\chi dx+\int_{\mathbb{R}^N}b(x)|\nabla u|^\alpha u^\beta\chi dx
+\int_{\mathbb{R}^N}\beta u^{\beta-1}|\nabla u|^p\chi
dx\leq\int_{\mathbb{R}^N}|\nabla u|^{p-1}u^\beta|\nabla\chi| dx.$$

By the hypothesis under $f$, there exists a $R_0>0$ (we can consider
this $R_0>0$ such that $a,b>0$ on $\mathbb{R}^N \setminus
  B_{R_0}$) such that $f(u(x)) \geq C u^q(x)$ and
$u(x) \geq 1$ for all $\vert x \vert \geq R_0$ for some $C>0$, since
$u(x) \to \infty$ as $\vert x \vert \to \infty$. That is,
$$
\begin{array}{cll}
\displaystyle C\int_{R_0\leq|x|\leq 2R}a(x)u^{\beta+q}\chi
dx&+&\displaystyle\int_{R_0\leq|x|\leq 2R}b(x)|\nabla u|^\alpha
u^{\beta}\chi dx +\displaystyle\beta\int_{R_0\leq|x|\leq 2R}
u^{\beta-1}|\nabla u|^p\chi
dx\leq  \\
\\
&&\displaystyle\int_{R_0\leq|x|\leq 2R}|\nabla
u|^{p-1}u^\beta|\nabla\chi| dx,~R>R_0,
\end{array}
$$
for some  $C>0$. Now, we can rewrite the above inequality as
\begin{equation}
\label{n1}
\begin{array}{cll}
 \displaystyle\int_{R_0\leq|x|\leq 2R}a(x)u^{\beta+q}\chi dx&+&\displaystyle\int_{R_0\leq|x|\leq 2R}b(x)|\nabla u|^\alpha u^{\beta}\chi dx
+\displaystyle\int_{R_0\leq|x|\leq 2R} u^{\beta-1}|\nabla u|^p\chi dx\leq \\
\\
&&\displaystyle\tilde{C}\int_{R_0\leq|x|\leq 2R}|\nabla
u|^{p-1}u^\beta|\nabla\chi| dx,
\end{array}
\end{equation}
where $\tilde{C}>0$ is a real constant depending of $C$ and $\beta$.

From now on, we going to consider two cases:
\smallskip

\noindent Case 1: $(i)$ holds. First, we note that $q>p-1$. So,
given $\tau\in (1+{1}/{q},1+{1}/{(p-1)})$ and considering $\tau'>1$
satisfying ${1}/{\tau'}+{1}/{\tau}=1$, we can use the inequality of
Young, to obtain
\begin{equation}
\label{n2}
\begin{array}{cll}
 &&\displaystyle\int_{R_0\leq|x|\leq 2R} a(x)^{\frac{1}{\tau'}}u^{(\beta+\frac{q}{\tau'}-\frac{1}{\tau})}|\nabla u|^{\frac{p}{\tau}}\chi dx
= \displaystyle\int_{R_0\leq|x|\leq
2R}\Big(a(x)u^{{\beta+q}}\chi\Big)^{\frac{1}{\tau'}}\Big(u^{{\beta-1}}
|\nabla u|^{{p}}\chi\Big)^{\frac{1}{\tau}}dx\\
\\
&\leq& \displaystyle\frac{1}{\tau'}\int_{R_0\leq|x|\leq
2R}a(x)u^{\beta+q}\chi dx+\frac{1}{\tau}\int_{R_0\leq|x|\leq
2R}u^{\beta-1}
|\nabla u|^p\chi dx\\
\\
&\leq& \displaystyle\int_{R_0\leq|x|\leq 2R}a(x)u^{\beta+q}\chi
dx+\int_{R_0\leq|x|\leq 2R}u^{\beta-1}|\nabla u|^p\chi dx.
\end{array}
\end{equation}

So, from (\ref{n1}) and (\ref{n2}), we get
$$
\int_{R_0\leq|x|\leq
2R}a(x)^{\frac{1}{\tau'}}u^{(\beta+\frac{q}{\tau'}-\frac{1}{\tau})}|\nabla
u|^{\frac{p}{\tau}}\chi dx+ \int_{R_0\leq|x|\leq 2R}b(x)|\nabla
u|^\alpha u^{\beta}\chi dx \leq\tilde{C} \int_{R_0\leq|x|\leq
2R}|\nabla u|^{p-1}u^\beta|\nabla\chi|dx.
$$
and so, we have
\begin{enumerate}
  \item [($I$)] $\displaystyle\int_{R_0\leq|x|\leq
2R}a(x)^{\frac{1}{\tau'}}u^{(\beta+\frac{q}{\tau'}-\frac{1}{\tau})}|\nabla
u|^{\frac{p}{\tau}}\chi dx \leq\tilde{C} \int_{R_0\leq|x|\leq
2R}|\nabla
u|^{p-1}u^\beta|\nabla\chi|dx $ \\
\\
\noindent and
  \item [($II$)] $\displaystyle\int_{R_0\leq|x|\leq 2R}b(x)|\nabla u|^\alpha u^{\beta}\chi dx\leq\tilde{C}
\int_{R_0\leq|x|\leq 2R}|\nabla u|^{p-1}u^\beta|\nabla\chi|dx. $
\end{enumerate}
\smallskip

\noindent Assume $(I)$ holds.  So, letting
$$m=\frac{\beta\tau\tau'+q\tau-\tau'}{\beta\tau\tau'}>1~\mbox{and}~n=\frac{\beta\tau\tau'+q\tau-\tau'} {q\tau-\tau'}>1,$$
we have $1/m + 1/n = 1$ and by H\"older inequality, it follows that
\begin{equation}
\label{n0}
\begin{array}{l}
\displaystyle\int_{R_0\leq|x|\leq 2R}|\nabla
u|^{p-1}u^\beta|\nabla\chi|dx= \int_{R_0\leq|x|\leq
2R}a(x)^{\frac{1}{\tau'm}}u^\beta|\nabla
u|^{p-1}\chi^{\frac{1}{m}}a(x)^{-\frac{1}{\tau'm}}\chi^{-\frac{1}{m}}|
\nabla\chi|dx \leq\\
\displaystyle C\left(\int_{R_0\leq|x|\leq
2R}a(x)^{\frac{1}{\tau'}}u^{(\beta+\frac{q}{\tau'}-\frac{1}{\tau})}|\nabla
u|^{(p-1)m}\chi dx\right)^ {\frac{1}{m}}\!\!\left(\int_{R\leq|x|\leq
2R}a(x)^{-\frac{n}{\tau'm}}\chi^{-\frac{n}{m}}|\nabla\chi|^{n}dx\right)^{\frac{1}{n}}.
\end{array}
\end{equation}

Now, choosing
$$\beta=\beta(\tau)=\frac{(p-1)(q\tau-\tau')}{\tau'(p-\tau p+\tau)}, ~\tau>1$$
and noticing that $(p-1)m={p}/{\tau}$, we get
$$
\begin{array}{l}
\displaystyle\int_{R_0\leq|x|\leq
2R}a(x)^{\frac{1}{\tau'}}u^{(\beta+\frac{q}{\tau'}-\frac{1}{\tau})}|\nabla
u|^{\frac{p}{\tau}}\chi dx \leq\\
\\
\displaystyle C\left(\int_{R_0\leq|x|\leq
2R}a(x)^{\frac{1}{\tau'}}u^{(\beta+\frac{q}{\tau'}-\frac{1}{\tau})}|\nabla
u|^{\frac{p}{\tau}}\chi dx\right)^
{\frac{1}{m}}\left(\int_{R\leq|x|\leq 2R}a(x)^{-\frac{\beta \tau}{q
\tau-\tau'}}\chi^{-\frac{\beta\tau\tau'}{q\tau-\tau'}}|\nabla\chi|^{n}
dx\right)^{\frac{1}{n}}.
\end{array}
$$

That is
\begin{equation}
\label{n3}
 \int_{R_0\leq|x|\leq 2R}a(x)^{\frac{1}{\tau'}}u^{(\beta+\frac{q}{\tau'}-\frac{1}{\tau})}|\nabla u|^{\frac{p}{\tau}}\chi dx\leq
C\int_{R\leq|x|\leq
2R}a(x)^{-\frac{\beta\tau}{q\tau-\tau'}}\chi^{-\frac{\beta
\tau\tau'}{q\tau-\tau'}}|\nabla\chi|^{n}dx.
\end{equation}

Besides this, we noting that
$$\chi^{-\frac{\beta\tau\tau'}{q\tau-\tau'}}=\xi^{-\frac{\beta\tau\tau'}{q\tau-\tau'}\mu}~\mbox{and}~|\nabla\chi|=\mu\xi^{\mu-1}|\nabla\xi|\leq
\mu\xi^{\mu-1}R^{-1},~x\in \mathbb{R}^N,
$$
we get for $\mu$ large enough, that

$$\chi^{-\frac{\beta\tau\tau'}{q\tau-\tau'}}|\nabla\chi|^n\leq\mu^n\xi^{(\mu-1)n-\frac{\beta\tau\tau'}{q\tau-\tau'}\mu}R^{-n}\leq
C_{\mu}R^{-n},\ \mbox{for all}\ x \in \mathbb{R}^N,$$ for some
$C_{\mu}>0$.

Now, fixing a such $\mu>1$, we have that
\begin{equation}
\label{n4}\int_{R\leq|x|\leq
2R}a(x)^{-\frac{\beta\tau}{q\tau-\tau'}}\chi^{-\frac{\beta
\tau\tau'}{q\tau-\tau'}}|\nabla\chi|^{n}dx\leq
C_{\mu}R^{-\frac{p}{p-\tau p+\tau}} \int_{R\leq|x|\leq
2R}a(x)^{-\frac{(\tau-1)(p-1)}{p-\tau p+\tau}}dx.
\end{equation}

Now, given $\theta\in(p-1,q)$ we can take a $\tau = \tau_{\theta}\in
(1+{1}/{q},1+{1}/{(p-1)})$ such that $\theta=q(p-1)(\tau-1)$. So,
from (\ref{n4}), we have
\begin{equation}
\label{n51}\int_{R\leq|x|\leq
2R}a(x)^{-\frac{\beta\tau}{q\tau-\tau'}}\chi^{-\frac{\beta
\tau\tau'}{q\tau-\tau'}}|\nabla\chi|^{n}dx \leq
R^{\frac{pq}{\theta-q}}\int_{R\leq|x|\leq
2R}a(x)^{\frac{\theta}{\theta-q}}dx,~R>R_0.
\end{equation}

Hence, it follows from (\ref{n3}), (\ref{n51}) and the hypothesis
($i_1$), that
\begin{equation}
\label{n52}\int_{|x|\geq
R_0}a(x)^{\frac{1}{\tau'}}u^{(\beta+\frac{q}{\tau'}-\frac{1}{\tau})}|\nabla
u|^{\frac{p}{\tau}}dx<\infty.
\end{equation}

Now returning in (\ref{n0}), rewriting its last integrals with the
domain ${R_0\leq|x|\leq 2R}$ instead of ${R\leq|x|\leq 2R}$ and
using (\ref{n51}), we get
$$
\begin{array}{l}
\displaystyle\int_{R_0\leq|x|\leq
2R}a(x)^{\frac{1}{\tau'}}u^{(\beta+\frac{q}{\tau'}-\frac{1}{\tau})}|\nabla
u|^{\frac{p}{\tau}}\chi dx \leq\\
\\
\displaystyle C\left(\int_{R\leq|x|\leq
2R}a(x)^{\frac{1}{\tau'}}u^{(\beta+\frac{q}{\tau'}-\frac{1}{\tau})}|\nabla
u|^{\frac{p}{\tau}}\chi dx\right)^
{\frac{1}{m}}\left(R^{\frac{pq}{\theta-q}}\int_{R\leq|x|\leq
2R}a(x)^{\frac{\theta}{\theta-q}}dx\right)^{\frac{1}{n}}.
\end{array}
$$

Now, it follows from (\ref{n52}) and the hypothesis, that
$$\displaystyle\int_{|x|\geq R_0}a(x)^{\frac{1}{\tau'}}u^{(\beta+\frac{q}{\tau'}-\frac{1}{\tau})}|\nabla
u|^{\frac{p}{\tau}} dx=0,$$ that is, $u(x)=c$, for all $x \in
\mathbb{R}^N \setminus B_{R_0}$, for some real constant $c>0$. This
is impossible, because $u(x) \to \infty$ as $\vert x \vert \to
\infty$.
\smallskip

\noindent Assume $(II)$ holds. First, we note that we can take
$\beta=0$ (return at the beginning of the proof and taking just
$\chi$ as a test function in place of $u^{\beta} \chi$). Now,
letting $m={\alpha}/({p-1})>1$ and $n={\alpha}/({\alpha-p+1})>1$, we
have $1/m +1/n=1$ and
$$
\begin{array}{lll}
\displaystyle\int_{R_0\leq|x|\leq 2R}b(x)|\nabla u|^\alpha\chi
dx&\leq& \displaystyle\tilde{C}\int_{R_0\leq|x|\leq
2R}[b(x)\chi]^{\frac{1}{m}}[b(x)\chi]^{-\frac{1}{m}}
|\nabla u|^{p-1}|\nabla\chi| dx\\
\\
&\leq&\displaystyle C\Big(\int_{R_0\leq|x|\leq 2R}b(x)\chi|\nabla
u|^\alpha dx\Big)^{\frac{1}{m}}\Big(\int_{R\leq|x|\leq
2R}[b(x)\chi]^{-\frac{n}{m}} |\nabla\chi|^ndx\Big)^{\frac{1}{n}}.
\end{array}
$$
That is,
\begin{equation}
\label{n5} \int_{R_0\leq|x|\leq 2R}b(x)|\nabla u|^\alpha\chi dx\leq
C\int_{R\leq|x|\leq 2R}[b(x)\chi]^{-\frac{n}{m}} |\nabla\chi|^ndx.
\end{equation}

Since,
$$|\nabla\chi|^n\chi^{-\frac{n}{m}}=\mu^n\xi^{\mu-n}|\nabla\xi|^n\leq\mu^nR^{-n}\
\mbox{for all}\ \mu\geq n
$$
it follows from (\ref{n5}), that
$$
\int_{R_0\leq|x|\leq 2R}b(x)|\nabla u|^\alpha \chi dx\leq
CR^{-\frac{\alpha}{\alpha-p+1}}\int_{R\leq|x|\leq
2R}b(x)^{-\frac{p-1}{\alpha-p+1}}dx.
$$

Thus, it follows from the hypothesis,  that
$$\displaystyle  \int_{|x|\geq R_0}b(x)|\nabla u|^\alpha  dx<\infty$$
and the rest of the proof follows like the last case.
\smallskip

\noindent Case 2: $(ii)$ holds. Given $\tau\in(1,{\alpha}/({p-1}))$,
take $\tau'>1$ such that ${1}/{\tau}+{1}/{\tau'}=1$. So, we have
$$
\begin{array}{lll}
&&\displaystyle \int_{R_0\leq|x|\leq 2R}\Big(a(x) u^{{\beta+q}} \chi
\Big)^{\frac{1}{\tau'}}\Big(b(x) u^{{\beta}}|\nabla u|^{{\alpha}}
\chi\Big)^{\frac{1}{\tau}}dx\\
\\
&\leq&\displaystyle\frac{1}{\tau'}\int_{R_0\leq|x|\leq
2R}a(x)u^{\beta+q}\chi dx+\frac{1}{\tau}\int_{R_0\leq|x|\leq
2R}b(x)u^\beta
|\nabla u|^\alpha\chi dx\\
\\
&\displaystyle\leq&\displaystyle\int_{R_0\leq|x|\leq
2R}a(x)u^{\beta+q}\chi dx+\int_{R_0\leq|x|\leq 2R}b(x)u^\beta|\nabla
u|^\alpha\chi dx.
\end{array}
$$

Hence, it follows from (\ref{n1}) that
 \begin{equation} \label{n6}
\int_{R_0\leq|x|\leq
2R}a(x)^{\frac{1}{\tau'}}b(x)^{\frac{1}{\tau}}u^{(\beta+\frac{q}{\tau'})}\chi|\nabla
u|^{\frac{\alpha}{\tau}} dx\leq \tilde{C}\int_{R_0\leq|x|\leq
2R}|\nabla u|^{p-1}u^\beta|\nabla\chi|dx.
\end{equation}

Now, defining
$$\beta=\beta(\tau)=\frac{q(p-1)}{(\tau'-1)(\alpha-\tau p+\tau)},~~ m=\frac{\beta\tau'+q}{\beta\tau'}~~\mbox{and}~~n=\frac{\beta\tau'+q}{q},$$
we have $1/m + 1/n =1$ and
$$
\begin{array}{lll}
&&\displaystyle\int_{R_0\leq|x|\leq 2R}|\nabla
u|^{p-1}u^\beta|\nabla\chi|dx=\\
\\
 &&\displaystyle\int_{R_0\leq|x|\leq
2R}a(x)^{\frac{\beta}{\beta\tau'+q}}b(x)^{\frac{\beta\tau'}{(\beta\tau'+q)\tau}}u^\beta
\chi^{\frac{\beta\tau'}{\beta\tau'+q}}|\nabla
u|^{p-1}a(x)^{-\frac{\beta}{\beta\tau'+q}}b(x)^{-\frac{\beta\tau'}{(\beta\tau'
+q)\tau}}\chi^{-\frac{\beta\tau'}{\beta\tau'+q}}|\nabla\chi|dx\leq\\
\\
&&\displaystyle C\Big(\int_{R_0\leq|x|\leq
2R}a(x)^{\frac{1}{\tau'}}b(x)^{\frac{1}{\tau}}u^{(\beta+\frac{q}{\tau'})}\chi|\nabla
u|^{(p-1)m}dx\Big)^{\frac{1}{m}} \Big(\int_{R\leq|x|\leq
2R}a(x)^{-\frac{\beta}{q}}b(x)^{-\frac{\beta\tau'}{q\tau}}\chi^{-\frac{\beta\tau'}{q}}|\nabla
\chi|^{n} dx\Big)^{\frac{1}{n}}.
\end{array}$$

Since that $(p-1)m={\alpha}/{\tau}$, it follows that
$$
\begin{array}{lll}
&&\displaystyle\int_{R_0\leq|x|\leq 2R}|\nabla
u|^{p-1}u^\beta|\nabla\chi|dx\leq\\
\\
&&\displaystyle C\Big(\int_{R_0\leq|x|\leq
2R}a(x)^{\frac{1}{\tau'}}b(x)^{\frac{1}{\tau}}u^{(\beta+\frac{q}{\tau'})}\chi|\nabla
u|^{\frac{\alpha}{\tau}}dx\Big)^
{\frac{1}{m}}\Big(\int_{R\leq|x|\leq
2R}a(x)^{-\frac{\beta}{q}}b(x)^{-\frac{\beta\tau'}{q\tau}}\chi^{-\frac{\beta\tau'}{q}}
|\nabla \chi|^{n}dx\Big)^{\frac{1}{n}}.
\end{array}$$

Now, it follows from (\ref{n6}), that
$$\int_{R_0\leq|x|\leq 2R}a(x)^{\frac{1}{\tau'}}b(x)^{\frac{1}{\tau}}u^{(\beta+\frac{q}{\tau'})}\chi|\nabla u|^{\frac{\alpha}{\tau}}dx\leq
C\int_{R\leq|x|\leq
2R}a(x)^{-\frac{\beta}{q}}b(x)^{-\frac{\beta\tau'}{q\tau}}\chi^{-\frac{\beta\tau'}{q}}
|\nabla \chi|^{n}dx,$$ that is,

$$ \int_{R_0\leq|x|\leq
2R}a(x)^{\frac{1}{\tau'}}b(x)^{\frac{1}{\tau}}u^{(\beta+\frac{q}{\tau'})}\chi|\nabla
u|^{\frac{\alpha}{\tau}} dx\leq C\int_{R\leq|x|\leq
2R}a(x)^{-\frac{p-1}{(\tau'-1)(\alpha-\tau
p+\tau)}}b(x)^{-\frac{p-1}{\alpha-\tau p+\tau}}
\chi^{-\frac{\beta\tau'}{q}}|\nabla\chi|^ndx.
$$

Since
$$
\chi^{\frac{\beta\tau'}{q}}|\nabla\chi|^n\leq
CR^{-n}=CR^{-\frac{\alpha}{\alpha-\tau p+\tau}},
$$ it follows that
\begin{equation}
\label{n8}\int_{R_0\leq|x|\leq
2R}a(x)^{\frac{1}{\tau'}}b(x)^{\frac{1}{\tau}}u^{(\beta+\frac{q}{\tau'})}|\nabla
u|^{\frac{\alpha}{\tau}} \chi dx\leq CR^{-\frac{\alpha}{\alpha-\tau
p+\tau}}\int_{R\leq|x|\leq
2R}[a(x)^{(\tau-1)}b(x)]^{-\frac{p-1}{\alpha-\tau p+\tau}}dx.
\end{equation}

Now, given $\theta\in(p-1,\alpha)$, we take
$\tau=\tau_{\theta}\in(1,{\alpha}/({p-1}))$ such that
$\theta=(p-1)\tau$. Thus,
$$R^{-\frac{\alpha}{\alpha-\tau
p+\tau}}\int_{R\leq|x|\leq 2R}[a(x)^{(\tau-1)}b(x)]^
{-\frac{p-1}{\alpha-\tau p+\tau}}dx=R^{\frac{\alpha}{\theta-\alpha}}
\int_{R\leq|x|\leq 2R}[a(x)^{\frac{\theta}{p-1}-1}
b(x)]^{\frac{p-1}{\theta-\alpha}}dx.
$$

So, making $R \to \infty$ in (\ref{n8}), it follows from the
hypothesis ($ii$) and the last equality, that
$$\int_{|x|\geq R_0}a(x)^{\frac{1}{\tau'}}b(x)^{\frac{1}{\tau}}u^{(\beta+\frac{q}{\tau'})}|\nabla
u|^{\frac{\alpha}{\tau}}  dx< \infty.
$$

Now, making the same arguments after (\ref{n52}) we arrive in a
contradiction. So, we have finished the proof of theorem 1.2.\hfill
$\Box$

The next result, is a byproduct of Theorems \ref{teo2} and
\ref{teo4}. As a novelty there, in the existence issue, we have the
presence of the term $b$ that can change of sign. The case $b=0$ is
very studied. See for instance \cite{ye} with $p=2$
and \cite{JJ} for $1<p<\infty$ and references therein.

\begin{corollary}\label{cor3}
Assume that $a,b \in \L^{\infty}_{loc}(\mathbb{R}^N)$ with $a \geq
0$ and $\liminf_{t\to\infty}{f(t)}/{t^q}>0$ for some $q>p-1$. Then
the problem
\begin{equation} \label{b0}
 \left\{
\begin{array}{c}
 \Delta_pu = a(x)f(u) + b(x),\ \mbox{in}\ \mathbb{R}^N,\\
u \geq 0 ~\mbox{on}~ \mathbb{R}^N, ~~
u(x)\stackrel{|x|\to\infty}{\longrightarrow}\infty,
\end{array}
\right.
\end{equation}
admits:
\begin{enumerate}
  \item [$(i)$] at least one solution, if $\int_0^\infty \big(r^{1-N}\int_0^rs^{N-1}\displaystyle\max_{\vert x \vert=s}\{a(x),b^+(x)\}ds\big)^{\frac{1}{p-1}}dr<\infty$ holds, and
$a$ satisfies $(a_\Omega)$,
  \item [$(ii)$] no solution, if $b \geq 0$, $p < N$,
  $\int_1^\infty(r^{1-N}\int_0^rs^{N-1-\varepsilon}\displaystyle \min_{\vert x
  \vert=s}a(x)ds)^{\frac{1}{p-1}}dr=\infty$ and  there exists  $\lim_{|x|\to\infty}|x|^{p-\varepsilon}a(x)$, for some
  $\varepsilon\in(0,N-1)$.
\end{enumerate}
\end{corollary}

\noindent \textbf{Proof.} Assume $(i)$ holds. In this case, we know
that problem $(P_\rho)$, with $\rho(x)=\max\{a(x),b^+(x)\}$, $x \in
\mathbb{R}^N$, has a solution in $C^{1}(\mathbb{R}^{N})$. So, we
adapt the proof of Theorem \ref{teo1}, taking $\epsilon=\alpha=0$
and $\mu=1$, for the problem
$$
\left\{
\begin{array}{cc}
\Delta_p u = a(x) f(u)+b(x)~~ \mbox{in}~~B_n ,\\
u \geq 0~~ \mbox{in}~~  B_n,~~~ u(x) \stackrel{|x| \to
\partial B_n}{\longrightarrow} \infty.
\end{array}
\right.
$$

We point out that in this case, we can use a comparison principle of
Tolksdorf \cite{PT}, instead of that in \cite{GT}. Now, we apply the
proof of Case 1 of Theorem \ref{teo2}, taking again $\mu=1$. So
following the same proceedings, we get the problem (\ref{b0}) has a
solution.

Suppose $(ii)$ holds. Since
$$\int_1^\infty(r^{1-N}\int_0^rs^{N-1-\varepsilon}\displaystyle
\min_{\vert x
  \vert=s}a(x)ds)^{\frac{1}{p-1}}dr=\infty$$
holds, it follows by a direct computation that
$\lim_{|x|\to\infty}|x|^{\delta}a(x)=+
 \infty$, for all $p-\varepsilon < \delta <p$. Otherwise, there would be a $p-\varepsilon < \delta_o \leq p$
 such that
 $$\int_1^\infty\Big(r^{1-N}\int_0^rs^{N-1-\varepsilon}\min_{\vert x
  \vert=s}a(x)ds\Big)^{\frac{1}{p-1}}dr \leq C \int_1^{\infty} s^{(1-\varepsilon-\delta_o)/(p-1)}ds< \infty$$
 for some $C>0$. That is impossible by hypothesis.

 So, there exists a $R_0>0$ and
$\theta=\theta_{\varepsilon}\in({p-1},{q})$ such that
$N+({pq-\delta\theta})/({\theta-q})<0$ and
$a(x)^{{\theta}/({\theta-q})}\leq
C|x|^{-{\alpha\theta}/{(\theta-q)}}\ \ \mbox{for}\ |x|>R_0.$

Now, just computing, we get
$$
\limsup_{R\to\infty}R^{\frac{pq}{\theta-q}}\int_{R\leq|x|\leq
2R}a(x)^{\frac{\theta}{\theta-q}}dx=0<\infty.
$$
So, from Theorem \ref{teo4}, we have the claimed. \hfill $\Box$

\begin{corollary}\label{cor2} Assume  $\rho \in C(\mathbb{R}^N)$ with $\rho>0$  on
 $\mathbb{R}^N $ is such that
\begin{equation}
\label{cor5}
\limsup_{R\to\infty}R^{\frac{p}{\theta-1}}\int_{R\leq|x|\leq
2R}\rho(x)^{\frac{\theta}{\theta-1}}dx<\infty,\ \mbox{for some}\
\theta\in(0,1)
\end{equation}
holds. Then $(P_{\rho})$ has no solution in $C^1(\mathbb{R}^N)$.
\end{corollary}

\noindent\begin{proof} Consider a $\theta \in (0,1)$ satisfying
(\ref{cor5}). So, taking a $q>(p-1)/{\theta}>p-1$ and admitting that
problem $(P_{\rho})$ admits a solution in $C^1(\mathbb{R}^N)$, it
follows from the proof of Corollary 5.1 with $b=0$ and $f(t)=t^q$
that the problem
\begin{equation}
\label{1.53}
 \left\{
\begin{array}{c}
 \Delta_pv=\rho(x) v^q\ \ \mbox{in}\ \mathbb{R}^N,\\
v>0\ \mbox{in}\ \mathbb{R}^N,\ u(x)\stackrel{|x|
\to\infty}{\longrightarrow}\infty,
\end{array}
\right.
\end{equation}
has a solution in $C^1(\mathbb{R}^N)$.

Now, defining $\tilde{\theta}=q\theta$, we have
$\tilde{\theta}\in(p-1,q)$ and
$$\limsup_{R\to\infty}R^{\frac{pq}{\tilde{\theta}-q}}\int_{R\leq|x|\leq2R}\rho(x)^{\frac{\tilde{\theta}}{\tilde{\theta}-q}}dx=\limsup_{R\to\infty}R^{\frac{p}{\theta-1}}\int_{R\leq|x|\leq
2R}\rho(x)^{\frac{\theta}{\theta-1}}dx<\infty.$$ So,  by Theorem
1.2, it follows that problem (\ref{1.53}) does not have solution in
$C^1(\mathbb{R}^N)$, but this a contradiction. \hfill $\Box$
\end{proof}

 As examples of non-null and non-negative
potentials $\rho$ and $b$ satisfying Theorem \ref{teo4} and
Corollary $\ref{cor2}$ ($\rho=a$ in Theorem \ref{teo4}),  we have
(the first two cases below satisfy Theorem \ref{teo4}-$(i_1)$ and
Corollary $\ref{cor2}$ and third case satisfies Theorem
\ref{teo4}-$(i_2)$).
\begin{enumerate}
  \item [$(i)$]$\rho\in
L^\infty_{loc}(\mathbb{R}^N)$ such that
$\liminf_{|x|\to\infty}|x|^\delta\rho(x)>0$ for some either $\delta
< p$ or $\delta=p\geq N$,
 \item [$(ii)$] $\rho:\mathbb{R}^N \to [0,\infty)$ is continuous function with
 $\lim_{\vert x \vert \to \infty} \vert x \vert^{p-\varepsilon}\tilde{\rho}(\vert x \vert)\geq 0$ satisfying
$$
 \int_1^\infty
 \Big(r^{1-N}\int_0^rs^{N-1-\varepsilon}\tilde{\rho}(s)ds\Big)^{\frac{1}{p-1}}dr=\infty,
$$
 for some $\varepsilon \in (0,N-1)$, where $\tilde{\rho}(r)=\min_{\vert x\vert=r}\rho(x)$ and $1<p<N$,
 \item [$(iii)$]  $b\in
L^\infty_{loc}(\mathbb{R}^N)$ is such that
$\liminf_{|x|\to\infty}|x|^\delta b(x)>0$ for some
 $-\infty<\delta<N-{\alpha(N-1)}/{(p-1)}$.
\end{enumerate}


\end{document}